\documentclass[11pt, draft]{amsart}
\usepackage{amssymb, amstext, amscd, amsmath, amssymb}
\usepackage{mathtools, xypic, paralist, color, dsfont, rotating}
\usepackage{verbatim}
\usepackage{enumerate}

\numberwithin{equation}{section}

\makeatletter
\def\@cite#1#2{{\m@th\upshape\bfseries%
[{#1\if@tempswa{\m@th\upshape\mdseries, #2}\fi}]}}
\makeatother
\theoremstyle{plain}
\newtheorem{theorem}{Theorem}[section]
\newtheorem{corollary}[theorem]{Corollary}
\newtheorem{proposition}[theorem]{Proposition}

\theoremstyle{definition}
\newtheorem{definition}[theorem]{Definition}
\newtheorem{example}[theorem]{Example}
\newtheorem{examples}[theorem]{Examples}

\newtheorem{remark}[theorem]{Remark}

\newtheorem*{acknow}{Acknowledgements}
\theoremstyle{remark}

\mathtoolsset{centercolon}
  \newcommand{\A}{{\mathcal{A}}}
  \newcommand{\B}{{\mathcal{B}}}

\renewcommand{\H}{{\mathcal{H}}}

  \newcommand{\K}{{\mathcal{K}}}
\renewcommand{\L}{{\mathcal{L}}}
  \newcommand{\M}{{\mathcal{M}}}
  
\renewcommand{\O}{{\mathcal{O}}}

  \newcommand{\R}{{\mathcal{R}}}

\newcommand{\eps}{\varepsilon}
\def\al{\alpha}
\def\be{\beta}

\def\de{\delta}

\def\la{\lambda}

\def\om{\omega}

\def\si{\sigma}

\newcommand\vphi{\varphi}

\newcommand{\bA}{\mathbb{A}}
\newcommand{\bB}{\mathbb{B}}
\newcommand{\bC}{\mathbb{C}}

\newcommand{\bF}{\mathbb{F}}
\newcommand{\bH}{\mathbb{H}}

\newcommand{\bT}{\mathbb{T}}
\newcommand{\bZ}{\mathbb{Z}}
\newcommand{\bR}{\mathbb{R}}

\newcommand{\Bi}{{\mathbf{i}}}

\newcommand{\Bl}{{\mathbf{l}}}

\newcommand{\Br}{{\mathbf{r}}}

\newcommand{\foral}{\text{ for all }}
\newcommand{\qand}{\quad\text{and}\quad}

\newcommand{\qiff}{\quad\text{if and only if}\quad}
\newcommand{\qfor}{\quad\text{for}\ }

\newcommand{\ca}{\mathrm{C}^*}

\newcommand{\ol}{\overline}

\newcommand{\wh}{\widehat}

\newcommand{\ad}{\operatorname{ad}}
\newcommand{\Alg}{\operatorname{Alg}}

\newcommand{\Aut}{\operatorname{Aut}}
\newcommand{\alg}{\operatorname{alg}}

\newcommand{\diag}{\operatorname{diag}}

\newcommand{\dist}{\operatorname{dist}}

\newcommand{\End}{\operatorname{End}}

\newcommand{\id}{{\operatorname{id}}}

\newcommand{\Lat}{\operatorname{Lat}}
\newcommand{\mt}{\emptyset}

\newcommand{\Ref}{\operatorname{Ref}}

\newcommand{\spn}{\operatorname{span}}

\newcommand{\sca}[1]{\left\langle#1\right\rangle} 
\newcommand{\nor}[1]{\left\Vert #1\right\Vert} 
\newcommand{\bo}[1]{\mathbf{#1}} 

\newcommand{\scp}[3]{#1 \, \ol{\times}_{#2} \, #3}

\newcommand{\un}[1]{{\underline{#1}}} 

\addtocontents{toc}{\protect\setcounter{tocdepth}{1}}

\begin{document}

\title[Free Multivariate w*-Semicrossed Products]{Free Multivariate w*-Semicrossed Products: Reflexivity and the Bicommutant Property}

\author[R. T. Bickerton]{Robert T. Bickerton}
\address{School of Mathematics and Statistics\\ Newcastle University\\ Newcastle upon Tyne\\ NE1 7RU\\ UK}
\email{r.bickerton@ncl.ac.uk}

\author[E.T.A. Kakariadis]{Evgenios T.A. Kakariadis}
\address{School of Mathematics and Statistics\\ Newcastle University\\ Newcastle upon Tyne\\ NE1 7RU\\ UK}
\email{evgenios.kakariadis@ncl.ac.uk}



\maketitle

\vspace{-.5cm}
\begin{center}
{\small \emph{Dedicated to the memory of Donald E. Sarason.}}
\end{center}

\begin{abstract}
We study w*-semicrossed products over actions of the free semigroup and the free abelian semigroup on (possibly non-selfadjoint) w*-closed algebras.
We show that they are reflexive when the dynamics are implemented by uniformly bounded families of invertible row operators.
Combining with results of Helmer we derive that w*-semicrossed products of factors (on a separable Hilbert space) are reflexive.
Furthermore we show that w*-semicrossed products of automorphic actions on maximal abelian selfadjoint algebras are reflexive.
In all cases we prove that the w*-semicrossed products have the bicommutant property if and only if the ambient algebra of the dynamics does also.
\end{abstract}

\section{Introduction}\label{S:intro}

Reflexivity and the bicommutant property are closely related to invariant subspaces problems.
A w*-closed algebra $\A$ is \emph{reflexive} if it coincides with the algebra that leaves invariant the invariant subspaces of $\A$.
It is said to have the \emph{bicommutant property} if it coincides with its bicommutant $\A''$.
Von Neumann algebras are reflexive and have the bicommutant property, however this seems to be too crude to be the prototype.
Results are considerably harder to get for nonselfadjoint algebras.
For example $\A^{(\infty)}$ is always reflexive
but it may differ from $(\A^{(\infty)})''$, e.g. when $\A \neq \A''$.
Arveson \cite{Arv75} also introduced a function $\beta$ to measure reflexivity.
An algebra $\A$ is \emph{hyper-reflexive} if $\beta$ is equivalent to the distance function from $\A$.
A remarkable result of Bercovici \cite{Ber98} asserts that every wot-closed algebra whose commutant contains two isometries with orthogonal ranges is hyper-reflexive.

The reflexivity term is attributed to Halmos and it was first used by Radjavi-Rosenthal \cite{RR69}. 
It is considered as Noncommutative Spectral Synthesis in conjunction with synthesis problems in commutative Harmonic Analysis, and it offers a systematic way of reconstructing an algebra from a set of invariant subspaces; see the excellent exposition of Arveson \cite{Arv84}.
The first result regarding reflexivity concerns the Hardy algebra of the disc and it was proved by Sarason \cite{Sar66}.
It inspired a great amount of subsequent  research, e.g. Radjavi-Rosenthal \cite{RR73}, including the seminal work of Arveson \cite{Arv74} on CSL algebras.
Further examples include the important class of nest algebras \cite{Dav88}, the $\mathbb{H}^p$ Hardy algebras examined by Peligrad \cite{Pel80}, and algebras of commuting isometries or tensor products with the Hardy algebras studied by Ptak \cite{Pta86}. 
Algebras related to the \emph{free semigroup} $\bF_+^d$ were examined in a number of papers by Arias and Popescu \cite{AP95, Pop91}, Davidson, Katsoulis and Pitts \cite{DKP01, DP99}, Kennedy \cite{Ken09} and Fuller-Kennedy \cite{FK13}.
In far more generality, free semigroupoid algebras were also tackled by Kribs-Power \cite{KP04}.
Representations of the Heisenberg semigroup were recently studied by Anoussis-Katavolos-Todorov \cite{AKT12}.

Algebras related to dynamical systems 
(sometimes appearing as ``analytic crossed products'' in older papers) 
were considered by McAsey-Muhly-Saito \cite{MMS79}, Katavolos-Power \cite{KP02} and Kastis-Power \cite{KPo15}.
One-variable systems were further examined by the second author \cite{Kak09}. 
His work was extended by Helmer \cite{Hel14} to the much broader context of Hardy algebras of W*-correspondences in the sense of Muhly-Solel \cite{MS04}, and by Peligrad \cite{Pel14} to flows on von Neumann algebras.
Essential properties of the algebras of \cite{Kak09} were explored by Hasegawa \cite{Has15}.

The term of ``analytic crossed products'' has now been replaced by that of ``semicrossed products''.
In the last fifty years there has been a systematic approach, especially for their norm-closed variants.
The list of references is substantially long to be included here and the reader may refer to \cite{DFK13}.
We follow the work of the second author with Peters \cite{KPe15} and with Davidson and Fuller \cite{DFK14} and we interpret \emph{a} semicrossed product as an algebra densely spanned by generalized analytic polynomials subject to a set of covariance relations.
From the study in \cite{DFK14} it appears that semicrossed products over $\bF_+^d$ and $\bZ_+^d$ are the most tractable as the semigroups are finitely generated.
Therefore it is natural to examine their w*-closed variants, i.e. the \emph{w*-semicrossed products} in the sense of \cite{Kak09}.
These algebras arise through a Fock construction and in this paper we study the reflexivity and the bicommutant property for this specific representation.

Additional motivation comes from the recent results of Helmer \cite{Hel14}.
An application of his results shows reflexivity of semicrossed products of Type II or III factors over $\bF_+^d$.
With some modifications the arguments of \cite{Hel14} apply for Type II or III factors over $\bZ_+^d$.
Here we wish to conclude this programme by considering endomorphisms of $\B(\H)$.
Thus we focus on actions of $\bF_+^d$ or $\bZ_+^d$ such that each generator is implemented by a Cuntz family.
However we do not restrict just on $\B(\H)$.
There exists a plethora of dynamics implemented by Cuntz families appearing previously in the works of Laca \cite{Lac93}, Courtney-Muhly-Schmidt \cite{CMS12} and the second author with Peters \cite{KPe15}.
They arise naturally and form generalizations of the Cuntz-Krieger odometer (Examples \ref{E:odometer}).

We underline that our setting accommodates $\bZ_+^d$-actions where the generators $\al_{\Bi}$ are implemented by unitaries but those may not lift to a unitary action of $\bZ_+^d$, i.e. the unitaries implementing the actions may not commute.
For example any two commuting automorphisms over $\B(\H)$ are implemented by two unitaries that satisfy a Weyl's relation and may not commute (see Example \ref{E:Weyl}).
By using results of Laca \cite{Lac93} we are able to determine when an automorphism of $\B(\H)$ commutes with specific endomorphisms induced by two Cuntz isometries (see Examples \ref{E:com Cuntz 1} and \ref{E:com Cuntz 2}).

Our main results on reflexivity appear in Corollaries \ref{C:cun ref} and \ref{C:com ref} and are summarized in the following statement.
If $n_i$ is the multiplicity of the Cuntz family implementing the $i$-th generator of the action then we define
\[
N := \sum_{i=1}^d n_i \text{ for } \bF_+^d\text{-systems} \qand M := \prod_{i=1}^d n_i \text{ for } \bZ_+^d\text{-systems}
\] 
for the \emph{capacity} of the systems.

\begin{theorem}[Corollary \ref{C:cun ref}, Corollary \ref{C:com ref}]
Let $\al$ be an action of $\bF_+^d$ or $\bZ_+^d$ on $\A$ such that each generator of $\al$ is implemented by a Cuntz family.
If the capacity of the system is greater than $1$ then the resulting w*-semicrossed products are (hereditarily) hyper-reflexive.
If the capacity of the system is $1$ and $\A$ is reflexive then the resulting w*-semicrossed products are reflexive.
\end{theorem}

In fact we manage to tackle actions implemented by invertible row operators that satisfy a uniform bound hypothesis (Theorem \ref{T:cun ref}, Theorem \ref{T:com ref}). 
We term these as \emph{uniformly bounded spatial actions}.

The strategy we follow for $\bF_+^d$-systems is to realize the w*-semicrossed product as a subspace of $\B(\H) \, \ol{\otimes} \, \L_N$ (Theorem \ref{T:ue LN}).
Here $\L_N$ denotes the free semigroup algebra generated by the Fock representation for the capacity $N$ of the system.
Notice that even when $d=1$ we manage to pass to (a subspace of) the tensor product $\B(\H) \, \ol{\otimes} \, \L_{n_1}$.
When $N \geq 2$, $\B(\H) \, \ol{\otimes} \, \L_N$ is hyper-reflexive and has property $\bA_1(1)$ by \cite{Ber98, DP98}.
Hence by results of Kraus-Larson \cite{KL86} and Davidson \cite{Dav87} it follows that $\B(\H) \, \ol{\otimes} \, \L_N$ is hereditarily hyper-reflexive.
When $N= 1$ then the result follows from \cite{Kak09}.
For the $\bZ_+^d$-cases we decompose the w*-semicrossed product along the directions (Proposition \ref{P:disintegrate}) and apply similar arguments for the last factor of such a decomposition.

The passage inside $\B(\H) \, \ol{\otimes} \, \L_N$ relies on the strange phenomenon that every system on $\B(\H)$ given by a Cuntz family of multiplicity $n_i$ is equivalent to the trivial action of $\bF_+^{n_i}$ on $\B(\H)$.
This was first observed by the second author with Katsoulis \cite{KK12} and with Peters \cite{KPe15}.
Surprisingly there is a strong connection with the fact that module sums over the Cuntz algebra do not attain a unique basis.
Gipson \cite{Gip14} attacks this problem effectively by introducing the notion of the invariant basis number.

In combination with \cite{Hel14} we encounter systems over any factor and automorphic systems over maximal abelian selfadjoint algebras (Corollaries \ref{C:masa ref}, \ref{C:factor ref}, \ref{C:masa com ref} and \ref{C:factor com ref}).
It appears that the arguments of Helmer \cite{Hel14} treat a wider class of dynamical systems.
We include this information in Theorems \ref{T:Hel} and \ref{T:Hel com}.
Alongside this we translate his reflexivity proof in our context.

We mention that our reflexivity results can be acquired without referring to hyper-reflexivity, when $\A$ is reflexive.
To this end we provide a straightforward proof of that $\B(\H) \, \ol{\otimes} \, \L_d$ is reflexive (Proposition \ref{P:spatial}).
The line of reasoning resembles to \cite{Kak09, KP04} and may find applications to other settings, e.g. algebras over weighted graphs of Kribs-Levene-Power \cite{KLP16}.

By applying \cite{KL86, Dav87} we get that the hyper-reflexivity constant in Theorems \ref{T:cun ref} and \ref{T:com ref} is at most $7 \cdot K^4$ when $N, M \geq 2$ (where $K$ is the uniform bound for the invertible row operators).
However it can be decreased further to $3 \cdot K^4$.
This follows by analyzing their commutant.
In each case we identify the commutant with a twisted w*-semicrossed product over the commutant (Theorems \ref{T:com cun} and \ref{T:com com}).
Such algebras were studied in the norm context by the second author with Peters \cite{KP14}.
They form the nonselfadjoint analogues of the twisted C*-crossed product introduced earlier by Cuntz \cite{Cun81}. 
The method of twisting for w*-closed algebras was explored for automorphic $\bZ_+$-actions in \cite{Kak09} and applies also for $\bZ_+^d$-actions here.
Twisting twice brings us back to the w*-semicrossed product over the bicommutant.
Therefore we obtain Corollaries \ref{C:bicom cun} and \ref{C:bicom com} that can be summarized in the following statement.

\begin{theorem}[Corollary \ref{C:bicom cun}, Corollary \ref{C:bicom com}]
Let $\al$ be an action of $\bF_+^d$ or $\bZ_+^d$ on a w*-closed algebra $\A$.
Suppose that each generator of $\al$ is implemented by a Cuntz family.
Then $\A$ has the bicommutant property if and only if any (and thus all) of the resulting w*-semicrossed products does so.
\end{theorem}

For our analysis we use a generalized Fej\'{e}r Lemma; details are given in Section \ref{S:pre}.
For directly showing the reflexivity of $\B(\H) \, \ol{\otimes} \, \L_d$ we use finite dimensional cyclic modules.
In Section \ref{S:dyn sys} we define the algebras that play the role of the w*-semicrossed products.
However the important feature in $\bF_+^d$ is the separation between left and right lower triangular operators.
Obviously this separation is redundant for $\bZ_+^d$.
The results about the commutant and reflexivity appear in Sections \ref{S:commutant} and \ref{S:reflexivity}, respectively.

We underline that $\bF_+^d$ and $\bZ_+^d$ are tractable due to their simple structure.
Another interesting class of algebras is formed by systems over the Heisenberg semigroup \cite{AKT12}.
We leave this class for a subsequent project.

\begin{acknow}
This paper is part of the Ph.D. thesis of the first author.
The first author acknowledges support from EPSRC for his Ph.D. studies (project Ref. No. EP/M50791X/1).

The second author would like to thank the Isaac Newton Institute for Mathematical Sciences, Cambridge, for support and hospitality during the programme ``Operator algebras: subfactors and their applications'' where work on this paper was undertaken. 
This work was supported by EPSRC grant no EP/K032208/1.

The authors would like to thank Matthew Kennedy for useful discussions on the $\bA_1$ property and Masaki Izumi for constructive discussions on commuting endomorphisms of $\B(\H)$. 
\end{acknow}

\section{Preliminaries}\label{S:pre}

For $d \in \bZ_+ \cup \{\infty\}$ we write $[d]:= \{1, \dots, d\}$ so that $[\infty] = \bZ_+$.
We highlight that $d$ is always finite in $\bZ_+^d$, but $d \in \{1, 2, \dots, \infty\}$ in $\bF_+^d$.
We will write $\mathfrak{f}_\mu$ for a symbol $\mathfrak{f}$ and a word $\mu = \mu_m \dots \mu_1 \in \bF_+^d$ to denote 
\[
\mathfrak{f}_\mu = \mathfrak{f}_{\mu_m} \cdots \mathfrak{f}_{\mu_1}.
\]
To avoid any ambiguity as to what $\mathfrak{f}_\mu^*$ means we use the notation $(\mathfrak{f}_\mu)^*$.

We use capital letters for operators acting on tensor product Hilbert spaces and small letters for operators acting on their factors.
This reduces considerably the usage of parentheses (which we omit) when the operators act on elementary tensor vectors.

Sums over an infinite family of operators are taken in the strong operator topology with respect to the net over finite subsets.
For the algebras $\A_1 \subseteq \B(\H_1)$ and $\A_2 \subseteq \B(\H_2)$ we write $\A_1 \, \ol{\otimes} \, \A_2$ for the w*-closure of their algebraic tensor product in $\B(\H_1 \otimes \H_2)$.

\subsection{Free semigroup operators}

We endow $\bF_+^d$ with a (left) partial ordering given by
\[
\nu \leq_l \mu \text{ if there exists } z \in \bF_+^d \text{ such that } \mu = z \nu.
\]
We want to keep track of whether we concatenate on the left or on the right and we also consider the right version
\[
\nu \leq_r \mu \text{ if there exists } z \in \bF_+^d \text{ such that } \mu = \nu z.
\]
For a word $\mu = \mu_k \dots \mu_1$ we write $\ol\mu := \mu_1 \dots \mu_k$ for the reversed word of $\mu$.
We define the left and right creation operators on $\ell^2(\bF_+^d)$ by
\[
\Bl_\mu e_w = e_{\mu w} \qand \Br_\nu e_w = e_{w \ol{\nu}}.
\]
Notice here that $\Br_\nu$ is the product $\Br_{\nu_{|\nu|}} \cdots \Br_{\nu_1}$ and it is the reverse notation of what is used in \cite{DP99}.
We write
\[
\L_d := \ol{\alg}^{\textup{wot}}\{\Bl_\mu \mid \mu \in \bF_+^d \}
\qand
\R_d := \ol{\alg}^{\textup{wot}}\{\Br_\mu \mid \mu \in \bF_+^d \}.
\]
Fej\'{e}r's Lemma (that follows) implies that there is no difference in considering the w*-topology instead, i.e.
\[
\L_d = \ol{\alg}^{\textup{w*}}\{\Bl_\mu \mid \mu \in \bF_+^d \}
\qand
\R_d = \ol{\alg}^{\textup{w*}}\{\Br_\mu \mid \mu \in \bF_+^d \}.
\]
The Fourier co-efficients in the w*- and the wot-setting coincide.

\begin{definition}
For $n \in \bZ_+ \cup \{\infty\}$ we say that a row operator $u = [u_1 \dots u_n \dots ] \in \B(\H \otimes \ell^2(n), \H)$ is \emph{invertible} if there exists a column operator $v = [v_1 \dots v_n \dots]^t \in \B(\H, \H \otimes \ell^2(n))$ such that
\[
v u = I_{\H \otimes \ell^2(n)} \qand \sum_{i \in [n]} u_i v_i = I_\H.
\]
\end{definition}

In particular we have that $v_i u_j = \de_{i,j} I_\H$
and that $\| \sum_{i \in F} u_i v_i \| \leq 1$ for any finite $F \subseteq [n]$.
Indeed if $P_F$ is the projection on $\H_F : = \ol{\spn} \{ \xi \otimes e_{i} \mid i \in F\}$ then
\[
\| \sum_{i \in F} u_i v_i h \| = \| \sum_{i \in [n]} u_i v_i P_F h \| = \nor{P_F h} = \nor{h}
\]
for all $h \in \H_F$.
We will consider actions implemented by invertible row operators subject to a uniform bound.

\begin{definition}\label{D:ubo defn}
Let $\{u_i\}_{i \in [d]}$ be a family of invertible row operators such that $u_i = [u_{i, j_i}]_{j_i \in [n_i]}$.
We say that $\{u_i\}_{i \in [d]}$ is \emph{uniformly bounded} if the operators
\[
\wh{u}_{\mu_m \dots \mu_1} = u_{\mu_m} \cdot (u_{\mu_{m-1}} \otimes I_{[n_{\mu_m}]}) \cdots (u_{\mu_1} \otimes I_{[n_{\mu_m} \cdots n_{\mu_2}]})
\]
and their inverses 
\[
\wh{v}_{\mu_1 \dots \mu_m} = (v_{\mu_1} \otimes I_{[n_{\mu_m} \cdots n_{\mu_2}]}) \cdots (v_{\mu_{m-1}} \otimes I_{[n_{\mu_m}]}) \cdot v_{\mu_m}
\]
are uniformly bounded with respect to $\mu_m \dots \mu_1 \in \bF_+^d$.
\end{definition}

Notice that if $n_i = 1$ for all $i \in [d]$ then $\wh{u}_{\mu_m \dots \mu_1} = u_{\mu_m} \dots u_{\mu_1} = u_\mu$.
In fact $\wh{u}_{\mu_m \dots \mu_1}$ is the row operator of all possible products of the $u_{\mu_i, j_{\nu_i}}$. 
Let us exhibit this construction with an example for finite multiplicities.

\begin{example}
Let the row operators $u_1$ and $u_2$ with $n_1 = 2$ and $n_2=3$.
Then the operator $\wh{u}_{12}$ is given by
\begin{align*}
\wh{u}_{12} 
& = 
u_1 \cdot (u_2 \otimes I_{n_1})
=
[u_{1,1} \;\; u_{1,2}] \cdot
\begin{bmatrix}
[u_{2,1} \;\; u_{2,2} \;\; u_{2,3}] & \\
& [u_{2,1} \;\; u_{2,2} \;\; u_{2,3}]
\end{bmatrix} \\
& =
[u_{1,1} u_{2,1} \;\; u_{1,1} u_{2,2} \;\; u_{1,1} u_{2,3}
\; u_{1,2} u_{2,1} \;\; u_{1,2} u_{2,2} \;\; u_{1,2} u_{2,3}].
\end{align*}
\end{example}

Similar remarks hold for $\bZ_+^d$.
Following the notation of \cite{DFK14} we write $\Bi$ for the elements in the canonical basis of $\bZ_+^d$ and
\[
\un{n} = (n_1, \dots, n_d) = \sum_{i=1}^d n_i \Bi
\]
for the elements in $\bZ_+^d$.
We use the same notation for elements in $\bR^d$.

The positive cone $\bZ_+^d$ induces a partial order in $\bZ^d$ in the sense that
\[
\un{n} \leq \un{m} \text{ if there exists } \un{z} \in \bZ_+^d \text{ such that } \un{m} = \un{z} + \un{n}.
\]
Due to commutativity there is no distinction between a left and a right version.
We define the creation operators in $\ell^2(\bZ_+^d)$ by $\Bl_{\un{m}} e_{\un{w}} = e_{\un{m} + \un{w}}$ and we write
\[
\bH^\infty(\bZ_+^d) := \ol{\alg}^{\textup{wot}}\{\Bl_{\un{m}} \mid \un{m} \in \bZ_+^d \}.
\]
Fej\'{e}r's Lemma (that follows) for $\bH^\infty(\bZ_+^d)$ implies that there is no difference in considering the w*-topology instead of the weak operator topology.

\subsection{Lower triangular operators}

We fix a Hilbert space $\H$ and consider $\H \otimes \ell^2(\bF_+^d)$.
Then $\B(\H \otimes \ell^2(\bF_+^d))$ admits a point-w*-continuous action induced by the unitaries
\[
U_s \xi \otimes e_w = e^{i|w|s} \xi \otimes e_w \foral \xi \otimes e_w,
\]
with $s \in [-\pi, \pi]$.
For $T \in \B(\H \otimes \ell^2(\bF_+^d))$ and $m \in \bZ_+$ \emph{the $m$-th Fourier coefficient} is then given by
\[
G_m(T) := \frac{1}{2\pi} \int_{-\pi}^{\pi} U_s T U_s^* e^{-ims} ds
\]
where the integral is considered in the w*-topology of $\B(\H \otimes \ell^2(\bF_+^d))$ for the Riemann sums.
An application of Fej\'{e}r's Lemma implies that the Cesaro sums
\[
\si_{n+1}(T) := \sum_{k=-n}^n (1 - \frac{|k|}{n+1}) G_k(T)
\]
converge to $T$ in the w*-topology.
For $T \in \B(\H \otimes \ell^2(\bF_+^d))$ we write $T_{\mu, \nu} \in \B(\H)$ for the $(\mu,\nu)$-entry given by
\[
\sca{T_{\mu,\nu} \xi, \eta} = \sca{T \xi \otimes e_\nu, \eta \otimes e_\mu} \foral \xi, \eta \in \H.
\]

\begin{definition}
An operator $T \in \B(\H \otimes \ell^2(\bF_+^d))$ is a \emph{left lower triangular operator} if $T_{\mu, \nu} = 0$ whenever $\nu \not<_l \mu$.
In a dual way $T \in \B(\H \otimes \ell^2(\bF_+^d))$ is a \emph{right lower triangular operator} if $T_{\mu, \nu} = 0$ whenever $\nu \not<_r \mu$.
\end{definition}

The next proposition shows that the Fourier co-efficients induce a graded structure on lower triangular operators.
For $\mu, \nu \in \bF_+^d$ we write
\[
L_\mu := I_\H \otimes \Bl_\mu \qand R_\nu := I_\H \otimes \Br_\nu.
\]
From now on we write $p_w$ for the projection of $\ell^2(\bF_+^d)$ to $e_w$.

\begin{proposition}\label{P:lower triangular}
If $T$ is a left lower triangular operator in $\B(\H \otimes \ell^2(\bF_+^d))$ then
\[
G_m(T) 
= 
\begin{cases} 
\sum_{|\mu| = m} \sum_{w \in \bF_+^d} L_\mu (T_{\mu w, w} \otimes p_w) & \text{ if } m \geq 0, \\
0 & \text{ if } m <0.
\end{cases}
\]
In a dual way if $T$ is a right lower triangular operator in $\B(\H \otimes \ell^2(\bF_+^d))$ then
\[
G_m(T) 
= 
\begin{cases} 
\sum_{|\mu| = m} \sum_{w \in \bF_+^d} R_\mu (T_{w \ol{\mu}, w} \otimes p_w) & \text{ if } m \geq 0, \\
0 & \text{ if } m <0.
\end{cases}
\]
\end{proposition}

\begin{proof}
We will consider just the left case.
The right case is proven in a similar way.
Fix $\nu, \nu' \in \bF_+^d$ and $\xi, \eta \in \H$.
Then we have that
\begin{align*}
\sca{G_m(T) \xi \otimes e_\nu, \eta \otimes e_{\nu'}}
& =
\frac{1}{2\pi} \int_{-\pi}^{\pi} \sca{T \xi \otimes e_{\nu}, \eta \otimes e_{\nu'}} e^{i(- m - |\nu| + |\nu'|)s} ds \\
& =
\de_{|\nu'|, m + |\nu|} \sca{T_{\nu', \nu} \xi, \eta}
\end{align*}
for all $T \in \B(\H \otimes \ell^2(\bF_+^d))$.
Hence $\sca{G_m(T) \xi \otimes e_\nu, \eta \otimes e_{\nu'}} = 0$ when $|\nu'| \neq m + |\nu|$.
Suppose that $T$ is in addition a left lower triangular operator.

First consider the case where $m < 0$.
If $|\nu'| = m + |\nu|$ then $|\nu'|<|\nu|$ and thus $\nu \not<_l \nu'$.
But then we get that $\sca{T_{\nu', \nu} \xi, \eta} = 0$ since $T$ is left lower triangular.
Hence $G_m(T) = 0$ when $m <0$.

Secondly for $m \geq 0$ we have that $\sca{T_{\nu', \nu} \xi, \eta} = 0$ whenever $\nu \not<_l \nu'$.
Consequently we obtain 
\begin{align*}
\sca{G_m(T) \xi \otimes e_\nu, \eta \otimes e_{\nu'}}
& =
\begin{cases} 
\sca{T_{\nu', \nu} \xi, \eta} & \text{ if } \nu \leq_l \nu' \text{ and } |\nu'| - |\nu| = m,\\
0 & \text{ otherwise}.
\end{cases}
\end{align*}
On the other hand we compute
\begin{align*}
\sum_{|\mu| = m} \sum_{w \in \bF_+^d} \sca{L_\mu (T_{\mu w, w} \otimes p_w) \xi \otimes e_\nu, \eta \otimes e_{\nu'}}
& =
\sum_{|\mu| = m} \de_{\mu\nu, \nu'} \sca{T_{\mu \nu, \nu}\xi, \eta} = \\
& \hspace{-2.2cm} =
\begin{cases} 
\sca{T_{\nu', \nu} \xi, \eta} & \text{ if } \nu \leq_l \nu' \text{ and } |\nu'| - |\nu| = m,\\
0 & \text{ otherwise},
\end{cases}
\end{align*}
and the proof is complete.
\end{proof}

Similar conclusions hold for $\B(\H \otimes \ell^2(\bZ_+^d))$ by considering the unitaries
\[
U_{\un{s}} \xi \otimes e_{\un{w}} = e^{i \sum_{i=1}^d w_i s_i} \xi \otimes e_{\un{w}} \foral \xi \otimes e_{\un{w}}
\]
for $\un{s} \in [-\pi, \pi]^d$, and the induced Fourier transform on $T \in \B(\H \otimes \ell^2(\bZ_+^d))$ given by
\[
G_{\un{m}}(T) := \frac{1}{(2\pi)^d} \int_{[-\pi,\pi]^d} U_{\un{s}} T U_{\un{s}}^* e^{-i\sum_{i=1}^d m_i s_i} d \un{s} \qfor \un{m} \in \bZ^d.
\]
This follows by extending the arguments concerning the Fourier transform on $\B(\H \otimes \ell^2)$ to the multi-dimensional case.
Alternatively one may see $G_{\un{m}}$ as the composition of appropriate inflations of $G_{m_i}$ along the directions of $\ell^2(\bZ_+^d)$.
For $T \in \B(\H \otimes \ell^2(\bZ_+^d))$ we write $T_{\un{m}, \un{n}} \in \B(\H)$ for the operator given by 
\[
\sca{T_{\un{m}, \un{n}} \xi ,\eta} = \sca{T \xi \otimes e_{\un{n}}, \eta \otimes e_{\un{m}}}.
\]

\begin{definition}
An operator $T \in \B(\H \otimes \ell^2(\bZ_+^d))$ is a \emph{lower triangular operator} if $T_{\un{m}, \un{n}} = 0$ whenever $\un{n} \not< \un{m}$.
\end{definition}

In analogy to $\bF_+^d$ we write $L_\un{m} = I_\H \otimes \Bl_{\un{m}}$ which is used for the graded structure induced by the Fourier co-efficients.
Now we write $p_{\un{w}}$ for the projection of $\ell^2(\bZ_+^d)$ to $e_{\un{w}}$.

\begin{proposition}\label{P:lower triangular com}
If $T$ is a lower triangular operator in $\B(\H \otimes \ell^2(\bZ_+^d))$ then
\[
G_{\un{m}}(T) 
= 
\begin{cases} 
\sum_{\un{w} \in \bZ_+^d} L_{\un{m}} (T_{\un{m} + \un{w}, \un{w}} \otimes p_{\un{w}}) & \text{ if } \un{m} \in \bZ_+^d, \\
0 & \text{ otherwise}.
\end{cases}
\]
\end{proposition}

\begin{proof}
Let $T$ be a lower triangular operator.
Then for $\un{n}, \un{n'} \in \bZ_+^d$ and $\xi, \eta \in \H$ we obtain
\begin{align*}
\sca{G_{\un{m}}(T) \xi \otimes e_{\un{n}}, \eta \otimes e_{\un{n'}}}
& = \\
& \hspace{-1.5cm} =
\frac{1}{(2\pi)^d} \int_{[-\pi,\pi]^d} \sca{T \xi \otimes e_{\un{n}}, \eta \otimes e_{\un{n'}}} e^{- i \sum_{i=1}^d (m_i + n_i - n_i') s_i} d \un{s} \\
& \hspace{-1.5cm} =
\de_{\un{n'}, \un{m} + \un{n}} \sca{T_{\un{n'}, \un{n}} \xi ,\eta}.
\end{align*}
If $\un{n'} = \un{m} + \un{n}$ for $\un{m} \notin \bZ_+^d$ then there exists an $i =1, \dots, d$ such that $n_i' < n_i$.
In this case $\un{n} \not< \un{n'}$ hence $T_{\un{n'}, \un{n}} = 0$ and thus $G_{\un{m}}(T) = 0$.
On the other hand if $\un{m} \in \bZ_+^d$ then a straightforward computation gives
\begin{align*}
 \sum_{\un{w} \in \bZ_+^d} \sca{L_{\un{m}} (T_{\un{m} + \un{w}} \otimes p_{\un{w}}) \xi \otimes e_{\un{n}}, \eta \otimes e_{\un{n'}}}
& = 
\sca{T_{\un{m} + \un{n}, \un{n}} \xi \otimes e_{\un{m} + \un{n}}, \eta \otimes e_{\un{n'}}} \\
& =
\de_{\un{n'}, \un{m} + \un{n}} \sca{T_{\un{m} + \un{n}, \un{n}} \xi, \eta}
\end{align*}
and the proof is complete.
\end{proof}

\subsection{Reflexivity and the $\bA_1$-property}

The reader is addressed to \cite{Con91} for full details.
In short, let $\A$ be a unital subalgebra of $\B(\H)$.
It will be called \emph{reflexive} if it coincides with 
\[
\Alg\Lat(\A) := \{T \in \B(\H) \mid (1-P)TP = 0 \foral P \in \Lat(\A)\}.
\]
Since $\A$ is unital we get that the $\Alg\Lat(\A)$ coincides with the \emph{reflexive cover of $\A$} in the sense of Loginov-Shulman \cite{LS75}, i.e. with
\[
\Ref(\A) := \{T \in \B(\H) \mid T\xi \in \ol{\A \xi} \foral \xi \in \H\}.
\]
The algebra $\A$ is called \emph{hereditarily reflexive} if every w*-closed subalgebra of $\A$ is reflexive.
It is immediate that (hereditary) reflexivity is preserved under similarities.

A w*-closed algebra $\A \subseteq \B(\H)$ is said to have the \emph{$\bA_1$ property} if every w*-continuous linear functional on $\A$ is given by a rank one functional.
It follows by \cite{LS75} that a w*-closed algebra $\A$ is hereditarily reflexive if and only if it is reflexive and has the $\bA_1$ property.
In particular $\A$ is said to have the \emph{$\bA_1(1)$ property} if for every $\eps >0$ and every w*-continuous linear functional $\phi$ on $\A$ there are vectors $h, g \in \H$ such that $\phi(a) = \sca{ah,g}$ and $\nor{h} \nor{g} \leq (1 + \eps) \nor{\phi}$.
The origins of the $\bA_1(1)$ property can be traced to the work of Brown \cite{Bro78}.

Davidson-Pitts \cite{DP98} show that the wot-closure of the algebraic tensor product of $\B(\H)$ with $\L_d$ satisfies the $\bA_1(1)$ property, when $d \geq 2$. 
Their arguments depend on the existence of two isometries with orthogonal ranges in the commutant; thus they also apply for the tensor product of $\B(\H)$ with $\R_d$.
It follows that the tensor products with respect to the weak operator topology coincide with those taken in the weak*-topology.

Arias and Popescu \cite{AP95} first showed that the algebras $\B(\H) \, \ol{\otimes} \, \L_d$ and $\B(\H) \, \ol{\otimes} \, \R_d$ are reflexive.
In fact they satisfy much stronger properties as we will soon present.
Their results concern the wot-versions and $d < \infty$.
Let us give here a direct proof that treats the $d = \infty$ case as well.

We require the following notation.
For $\la \in \bB_d$ and $w = w_m \dots w_1 \in \bF_+^d$ we write
\[
w(\la) = \la_{w_m} \cdots \la_{w_1}.
\]
In  \cite[Example 8]{AP95} and \cite[Theorem 2.6]{DP99} it has been observed that the eigenvectors of $\L_d^*$ are of the form
\[
\nu_\la = (1 - \nor{\la}^2)^{1/2} \sum_{w \in \bF_+^d} w(\la) e_w \qfor \la \in \mathbb{B}_d.
\]

\begin{proposition}\label{P:spatial} \cite{AP95}
The algebras  $\B(\H) \, \ol{\otimes} \, \L_d$ and $\B(\H) \, \ol{\otimes} \, \R_d$ are reflexive.
\end{proposition}

\begin{proof}
We just show that $\B(\H) \, \ol{\otimes} \, \L_d$ is reflexive.
Since the gauge action of $\B(\H \otimes \ell^2(\bF_+^d))$ restricts to a gauge action of $\B(\H) \, \ol{\otimes} \, \L_d$, it suffices to show that every $G_m(T)$ is in $\B(\H) \, \ol{\otimes} \, \L_d$ whenever $T$ is in $\Ref(\B(\H) \, \ol{\otimes} \, \L_d)$.

For $\xi, \eta \in \H$ and $\nu, \mu \in \bF_+^d$ there is a sequence $X_n \in \B(\H) \, \ol{\otimes} \, \L_d$ such that 
\[
\sca{T_{\mu, \nu}\xi, \eta}
=
\sca{T \xi \otimes e_\nu, \eta \otimes e_{\mu}}
=
\lim_n \sca{X_n \xi \otimes e_\nu, \eta \otimes e_{\mu}}
=
\lim_n \sca{[X_n]_{\mu, \nu}\xi, \eta}
\]
Taking $\nu \not<_l \mu$ gives that $T$ is left lower triangular as every $X_n$ is so.
Therefore it suffices to show that $T_{\mu z, z} = T_{\mu, \mt}$ for all $z \in \bF_+^d$.
Indeed, when this holds, we can write
\begin{align*}
G_m(T) & =
\begin{cases}
\sum_{|\mu| = m} L_\mu (T_{\mu, \mt} \otimes I) & \text{ if } m \geq 0, \\
0 & \text{ if } m < 0,
\end{cases}
\end{align*}
and thus $G_m(T) \in \B(\H) \, \ol{\otimes} \, \L_d$.
For convenience we use the notation
\[
T_{(\mu)} := L_\mu^* G_m(T) = \sum_{w \in \bF_+^d} T_{\mu w,w} \otimes p_w.
\]
We treat the cases $m=0$ and $m \geq 1$ separately.

\smallskip

\noindent $\bullet$ \textbf{The case $m=0$.} Let $z \in \bF_+^d$ and assume that $\{z_1, \dots, z_{|z|} \} \subseteq [d']$ for some finite $d'$.
If $d < \infty$ then take $d'=d$.
Let $\la \in \bB_{d'} \subseteq \bB_d$ such that $\la_i \neq 0$ for every $i \in [d']$, and consider the vector
\[
g = \sum_{w \in \bF_+^{d'}} w(\la) e_w.
\]
As $g$ is an eigenvector for $\L_d^*$ we have that $(L_\mu (x \otimes I))^* \xi \otimes g$ is in the closure of $\{y\xi \otimes g \mid y \in \B(\H)\}$.
Therefore for $\xi \in \H$ there exists a sequence $(x_n)$ in $\B(\H)$ such that
\begin{equation}\label{eq:1}
G_0(T)^* \xi \otimes g = \lim_n x_n^* \xi \otimes g.
\end{equation}
Hence for $\eta \in \H$ we get 
\begin{align*}
w(\la) \sca{\xi, T_{w,w}\eta} 
& =
\sca{\xi, T_{w,w} \eta} \sca{g, e_w} 
=
\sca{G_0(T)^* \xi \otimes g, \eta \otimes e_w} \\
& \hspace{-4pt} \stackrel{(\ref{eq:1})}{=}
\lim_n \sca{x_n^* \xi \otimes g, \eta \otimes e_w}
=
\lim_n \sca{\xi, x_n\eta} \sca{g, e_w} \\
& =
w(\la) \lim_n \sca{\xi, x_n\eta}.
\end{align*}
Applying for $w = \mt$ and $w = z$ we have that $T_{z,z} = T_{\mt, \mt}$ as $z(\la) \neq 0$.
Since $z$ was arbitrary we have that $G_0(T) = T_{\mt, \mt} \otimes I$.

\smallskip

\smallskip

\noindent $\bullet$ \textbf{The case $m \geq 1$.} 
We have to show that $T_{\mu z, z} = T_{\mu, \mt}$ for all $z \in \bF_+^d$ and $|\mu| = m$.
Notice that every $\mu$ of length $m$ can be written as $\mu = q i^\om$ for some $i \in [d]$ and $\om \geq 1$.
By symmetry it suffices to treat the case where $i=1$.

Hence in what follows we fix a word $\mu = q 1^\om$ of length $m = |q| + \om$ with
\[
\om \geq 1 \qand q = q_{|q|} \dots q_1 \text{ with } q_1 \neq 1 \text{ or } q = \mt.
\]
We will use induction on $|z|$.
To this end fix an $r \in (0,1)$.
For $w = w_{|w|} \dots w_1 \in \bF_+^d$ we write
\[
w(t) = w_t \dots w_1 \qfor t = 1, \dots, |w|.
\]

\noindent - For $|z| = 1$: 
First suppose that $q \neq \mt$.
Let the vectors
\[
v := e_\mt + \sum_{k=1}^\infty r^k e_{1^k}
\qand
\Bl_{q(t)} v
=
e_{q(t)} + \sum_{k=1}^\infty r^k e_{q(t) 1^k} 
\text{ for }
t = 1, \dots, |q|
\]
and fix $\xi \in \H$.
As $v$ is an eigenvector for $\L_d^*$ we get that $X^* \xi \otimes \Bl_{q} v$ is in the closure of
\[
\{ x \xi \otimes v + \sum_{t=1}^{|q|} x_{t} \xi \otimes \Bl_{q(t)} v \mid x, x_t \in \B(\H), t = 1, \dots, |q| \}
\]
for all $X \in \B(\H) \, \ol{\otimes} \, \L_d$.
Hence there are sequences $(x_n)$ and $(x_{t,n})$ in $\B(\H)$ such that
\begin{equation}\label{eq:2}
G_m(T)^* \xi \otimes \Bl_{q} v = \lim_n x_n^*\xi \otimes v + \sum_{t=1}^{|q|} x_{t,n}^*\xi \otimes \Bl_{q(t)} v.
\end{equation}
Furthermore for $|\mu'|=m$ we have that $(\Bl_{\mu'})^* \Bl_{q} v = \de_{\mu', \mu} r^{\om} v$.
Now for all $\eta \in \H$ and $z \in \bF_+^d$ we get that
\begin{align*}
\sca{G_m(T)^* \xi \otimes \Bl_{q} v, \eta \otimes e_z}
& =
r^{\om} \sca{\xi,  T_{q 1^\om z, z} \eta} \sca{v, e_z}.
\end{align*}
Every $\Bl_{q(t)} v$ is supported on $q(t) 1^k$ with $|q(t) 1^k| \geq t \geq 1$ and so $\sca{\Bl_{q(t)} v, e_\mt} = 0$ for all $t$.
By taking the inner product with $\eta \otimes e_\mt$ in equation (\ref{eq:2}) we get 
\[
r^{\om} \sca{\xi,  T_{q 1^\om, \mt} \eta} = \lim_n \sca{\xi, x_n \eta}.
\]
On the other hand the only vector of length $1$ in the support of $\Bl_{q(t)} v$ is achieved when $t=1$ and $k=0$, in which case it is $q(1) \neq 1$ by assumption.
Therefore by taking inner product with $\eta \otimes e_1$ in equation (\ref{eq:2}) we obtain
\begin{align*}
r^{\om + 1} \sca{\xi,  T_{q 1^\om 1, 1} \eta} 
=
\lim_n r \sca{\xi, x_n \eta}.
\end{align*}
Therefore $\sca{\xi,  T_{q 1^\om 1, 1} \eta} = \lim_n r^{-\om} \sca{\xi, x_n \eta} = \sca{\xi,  T_{q 1^\om, \mt} \eta}$ which implies that $T_{q 1^\om 1, 1} = T_{q 1^\om, \mt}$ when $q \neq \mt$.

On the other hand if $q = \mt$ then we repeat the above argument by substituting $\Bl_{q(t)} v$ with zeroes to get again that $ T_{1^\om 1, 1} = T_{1^\om,\mt}$.
In every case we have that $T_{\mu 1, 1} = T_{\mu, \mt}$.

Next we show that $T_{\mu 2, 2} = T_{\mu, \mt}$.
To this end let the vectors
\[
w = e_\mt + \sum_{k=1}^\infty r^k e_{2^k} \qand \Bl_{\mu(s)} w = e_{\mu(s)} + \sum_{k=1}^\infty r^k e_{\mu(s) 2^k} \text{ for } s = 1, \dots, m.
\]
As above, for $\xi \in \H$ there are sequences $(y_n)$ and $(y_{s,n})$ in $\B(\H)$ such that
\begin{equation}\label{eq:3}
G_m(T)^* \xi \otimes \Bl_{\mu} w
 =
\lim_n y_n^* \xi \otimes w + \sum_{s=1}^{m} y_{s,n}^* \xi \otimes \Bl_{\mu(s)} w
\end{equation}
since $w$ is an eigenvector of $\L_d^*$.
Notice here that $(\Bl_{\mu'})^*  \Bl_{\mu} w = \de_{\mu', \mu} w$ when $|\mu'| = m$.
Now for $\eta \in \H$ and $z \in \bF_+^d$ we get
\begin{align*}
\sca{G_m(T)^* \xi \otimes \Bl_{\mu} w, \eta \otimes e_z}
& =
\sca{\xi, T_{\mu z, z} \eta} \sca{w, e_z}.
\end{align*}
For $z = \mt$ we have that $\sca{\Bl_{\mu(s)} w, e_\mt} = 0$ for all $s \in [m]$ and therefore equation (\ref{eq:3}) gives
\[
\sca{\xi, T_{\mu, \mt} \eta} = \lim_n \sca{\xi, y_n \eta}.
\]
For $z = 2$ we have that $\sca{\Bl_{\mu(1)} w, e_2} = \sca{\Bl_{1} w, e_2} = 0$.
Moreover we have that $\sca{\Bl_{\mu(s)} w, e_2} = 0$ when $s \geq 2$.
Therefore equation (\ref{eq:3}) gives
\[
r \sca{\xi, T_{q 1^\om2 , 2} e_2} = \lim_n r \sca{\xi, y_n \eta}.
\]
As a consequence we have $\sca{\xi, T_{\mu 2, 2} e_2} = \sca{\xi, T_{\mu, \mt} \eta}$ and thus $T_{\mu 2, 2} = T_{\mu, \mt}$.
Applying for $i \in \{3, \dots, d\}$ yields $T_{\mu i, i} = T_{\mu, \mt}$ for all $i \in [d]$.

\smallskip

\noindent - Inductive hypothesis: Assume that $T_{q 1^\om z, z} = T_{q 1^\om, \mt}$ when $|z| \leq N$.
We will show that the same is true for words of length $N+1$.

Consider first the word $1 z$ with $|z| = N$.
Suppose that $q \neq \mt$ so that $q(1) \neq 1$.
We apply the same arguments for the vectors $\Br_z v$ and $\Br_z \Bl_{q(t)} v$ with $t = 1, \dots, |q|$.
Since $\Br_z$ commutes with every $\Bl_\nu$ we get that
\[
\Br_z (\Br_z)^* (\Bl_\nu)^* \Br_z v = \Br_z (\Bl_\nu)^* v
\qand
\Br_z (\Br_z)^* (\Bl_\nu)^* \Br_z \Bl_{q(t)} v = \Br_z (\Bl_\nu)^* \Bl_{q(t)} v.
\]
As every $R_z (R_z)^*$ commutes with every $x \otimes I$ for $x \in \B(\H)$, we have that for a fixed $\xi \in \H$ there are sequences $(x_n)$ and $(x_{t,n})$ in $\B(\H)$ such that
\begin{equation}\label{eq:4}
R_z (R_z)^* G_m(T)^* \xi \otimes \Br_z \Bl_{q} v = \lim_n x_n^*\xi \otimes \Br_z v + \sum_{t=1}^{|q|} x_{t,n}^*\xi \otimes \Br_z \Bl_{q(t)} v.
\end{equation}
Arguing as above for $\eta \otimes e_z$ and $\eta \otimes e_{1z}$ yields $\sca{\xi, T_{q 1^\om 1z, 1z} \eta} = \sca{\xi, T_{q 1^\om z, z} \eta}$.
Consequently $T_{q 1^\om 1z, 1z} = T_{q 1^\om z, z}$ which is $T_{q 1^\om, \mt}$ by the inductive hypothesis.

On the other hand if $q = \mt$ then we repeat the above arguments by substituting the $\Bl_{q(t)} v$ with zeroes.
Therefore in any case we have that $T_{\mu 1z, 1z} = T_{\mu, \mt}$.

For $2z$ with $|z| = N$ we take the vectors $\Br_z w$ and $\Br_z \Bl_{\mu(s)} w$ for $s \in [m]$.
Then for a fixed $\xi \in \H$ there are sequences $(y_n)$ and $(y_{s,n})$ in $\B(\H)$ such that
\begin{equation}\label{eq:5}
R_z (R_z)^* G_m(T)^* \xi \otimes \Br_z \Bl_{\mu} w
=
\lim_n y_n^*\xi \otimes \Br_z w + \sum_{s=1}^{m} y_{s,n}^* \xi \otimes \Br_z \Bl_{\mu(s)} w.
\end{equation}
Taking inner product with $\eta \otimes e_z$ and $\eta \otimes e_{2z}$ gives that $\sca{\xi , T_{\mu 2z, 2z} \eta} = \sca{\xi , T_{\mu z, z} \eta}$.
As $\eta$ and $\xi$ are arbitrary we then derive that $T_{\mu 2z, 2z} = T_{\mu z, z}$ which is $T_{\mu, \mt}$ by the inductive hypothesis.
Applying for $i \in \{3, \dots, d\}$ in place of $2$ gives the same conclusion, thus $T_{\mu i z, i z} = T_{\mu, \mt}$ for all $i \in [d]$ and $|z| = N$.
Induction then shows that $T_{\mu z, z} = T_{\mu, \mt}$ for all $z \in \bF_+^d$.
%
\end{proof}

\begin{remark}
Reflexivity of $\B(\H) \, \ol{\otimes} \, \bH^\infty(\bZ_+^d)$ can be proven along the same lines of reasoning by using the co-invariant subspaces $[x \xi \otimes g_{\Bi} \mid x \in \B(\H)]$ for the vectors
\[
g_{\Bi} = \sum_{k \in \bZ_+} r^k e_{k \Bi} \text{ with } r \in (0,1) \text{ and } i= 1, \dots, d.
\]
In fact one can show that $T$ is in $\B(\H) \, \ol{\otimes} \, \bH^\infty(\bZ_+^d)$ if and only if $T$ is lower triangular and $G_{\un{m}} = L_{\un{m}} (x_{\un{m}} \otimes I)$ for some $x_{\un{m}} \in \B(\H)$ whenever $\un{m} \in \bZ_+^d$.
The same holds for the tensor product of $\B(\H)$ with $\bH^\infty(\bZ_+^d)$ in the weak operator topology, inducing just one type of spatial tensor product.
\end{remark}

\subsection{Hyper-reflexivity}

Arveson \cite{Arv75} introduced a measurement for reflexivity.
For $\A \subseteq \B(\H)$ let the function $\be \colon \B(\H) \to \bR$ be given by
\[
\beta(T, \A) = \sup\{ \nor{(1-P)TP} \mid P \in \Lat(\A)\}.
\]
A w*-closed algebra $\A \subseteq \B(\H)$ is called \emph{hyper-reflexive} with distance constant at most $C$ if it satisfies
\[
\dist(T, \A) \leq C \beta(T, \A) \foral T \in \B(\H).
\]
Therefore hyper-reflexive algebras are reflexive.
Notice that $\beta(T, \A) \leq \dist(T, \A)$ always holds.

It follows that hyper-reflexivity can also be a hereditary property.
Kraus-Larson \cite{KL86} and Davidson \cite{Dav87} have shown that if $\A$ has the $\bA_1(1)$ property and is hyper-reflexive with distance constant at most $C$ then every w*-closed subspace of $\A$ is hyper-reflexive with distance constant at most $2C + 1$.

There is an alternative characterization of hyper-reflexivity through $\A_{\perp}$:
$\A$ is hyper-reflexive\footnote{\
Reflexivity is equivalent to $\A_{\perp}$ just being the closed linear span of its rank one functionals, e.g. \cite[Theorem 7.1]{Arv84}.}
if and only if for every $\phi \in \A_{\perp}$ there are rank one functionals $\phi_n \in \A_{\perp}$ such that $\phi = \sum_n \phi_n$ and $\sum_n \nor{\phi_n} < \infty$; e.g. \cite[Theorem 7.4]{Arv84}.
The hyper-reflexivity constant is at most $K$ when we achieve $\sum_n \nor{\phi_n} \leq K \cdot \nor{\phi}$ for $\phi = \sum_n \phi_n \in \A_{\perp}$ as in the representation above.
From this characterization it is readily verified that (hereditary) hyper-reflexivity is preserved under similarities.
Therefore if a similarity is given by an invertible $u$ then the hyper-reflexivity constant can change as much as $\nor{u}^2 \cdot \|u^{-1}\|^2$.

A remarkable result of Bercovici \cite{Ber98} asserts that a wot-closed algebra is hyper-reflexive with distance constant at most $3$ when its commutant contains two isometries with orthogonal ranges.
Consequently every w*-closed subalgebra of $\B(\H) \, \ol{\otimes} \, \L_d$ is hyper-reflexive with distance constant at most $3$, when $d \geq 2$, as its commutant contains $I_\H \, \ol{\otimes} \, \R_d$.

\section{Dynamical systems }\label{S:dyn sys}

We give the basic definitions of the w*-semicrossed products we will consider.
Henceforth we fix a w*-closed subalgebra $\A$ of $\B(\H)$.
Since we are working towards reflexivity and the bicommutant property we will assume that $\A$ is unital.
We write $\End(\A)$ for the unital w*-continuous completely bounded endomorphisms of $\A$, i.e. every $\al \in \End(\A)$ satisfies
\[
\nor{\al}_{cb} := \sup\{ \nor{ \al \otimes \id_n} \mid n \in \bZ_+\} < \infty.
\]

\subsection{Dynamical systems over $\bF_+^d$}\label{Ss:dyn free}

A (unital) w*-dynamical system denoted by $(\A, \{\al_i\}_{i \in [d]})$ consists of $d$ (unital) $\al_i \in \End(\A)$ such that
\[
\sup \{ \nor{\al_\mu} \mid \mu \in \bF_+^d\} < \infty.
\]
Given a w*-dynamical system $(\A, \{\al_i\}_{i \in [d]})$, we define two representations $\pi$ and $\ol{\pi}$ of $\A$ acting on $\K = \H \otimes \ell^2(\bF_+^d)$ by
\[
\pi(a) \xi \otimes e_\mu = \al_{\mu}(a)\xi \otimes e_\mu
\qand
\ol{\pi}(a) \xi \otimes e_\mu = \al_{\ol{\mu}}(a) \xi \otimes e_\mu.
\]
We need this distinction as the $\al_i$ induce both a homomorphism and an anti-homomorphism of $\bF_+^d$ in $\End(\A)$.
Note that $\pi(a)$ and $\ol{\pi}(a)$ are indeed in $\B(\K)$ as the $\al_\mu$ are uniformly bounded. 

\begin{definition}
Let $(\A, \{\al_i\}_{i \in [d]})$ be a w*-dynamical system.
We define the w*-semicrossed products
\[
\scp{\A}{\al}{\L_d} := \ol{\spn}^{\textup{w*}}\{L_\mu \ol{\pi}(a) \mid a \in \A, \mu \in \bF_+^d\}
\]
and
\[
\scp{\A}{\al}{\R_d} := \ol{\spn}^{\textup{w*}}\{R_\mu \pi(a) \mid a \in \A, \mu \in \bF_+^d\}.
\]
\end{definition}

The pairs $(\ol{\pi}, \{L_i\}_{i=1}^d)$ and $(\pi, \{R_i\}_{i=1}^d)$ satisfy the \emph{covariance relations}
\[
\ol{\pi}(a) L_i = L_i \ol{\pi} \al_i(a) \qand \pi(a) R_i = R_i \pi\al_i(a)
\]
for all $a \in \A$ and $i \in [d]$.
Indeed for every $w \in \bF_+^d$ we have that
\begin{align*}
\ol{\pi}(a) L_i \xi \otimes e_w
& =
\al_{\ol{i w}}(a) \xi \otimes e_{i w}
=
\al_{\ol{w}}\al_i(a)\xi \otimes e_{i w}
=
L_i \ol{\pi} \al_i(a) \xi \otimes e_w
\end{align*}
and similarly for the right version.
Consequently $\scp{\A}{\al}{\L_d}$ and $\scp{\A}{\al}{\R_d}$ are (unital) algebras.

The unitaries $U_s \in \B(\K)$ for $s \in [-\pi, \pi]$ induce a gauge action on $\scp{\A}{\al}{\L_d}$ since
\[
U_s \ol{\pi}(a) U_s^* = \ol{\pi}(a) \qand U_s L_\mu U_s^* = e^{i |\mu| s} L_\mu.
\]
Therefore Fej\'{e}r's Lemma implies that $T \in \scp{\A}{\al}{\L_d}$ if and only if $G_m(T) \in \scp{\A}{\al}{\L_d}$ for all $m \in \bZ$.
The same is true for $\scp{\A}{\al}{\R_d}$.

\begin{proposition}\label{P:Fc al}
Let $(\A, \{\al_i\}_{i \in [d]})$ be a unital w*-dynamical system.
Then an operator $T \in \B(\K)$ is in $\scp{\A}{\al}{\L_d}$ if and only if it is left lower triangular and
\[
G_m(T) = \sum_{|\mu| = m} L_\mu \ol{\pi}(a_\mu) \qfor a_\mu \in \A
\]
for all $m \in \bZ_+$.
Similarly an operator $T \in \B(\K)$ is in $\scp{\A}{\al}{\R_d}$ if and only if it is right lower triangular and
\[
G_m(T) = \sum_{|\mu| = m} R_\mu \pi(a_\mu) \qfor a_\mu \in \A
\]
for all $m \in \bZ_+$.
\end{proposition}

\begin{proof}
We will just show the left case.
First notice that if $T = L_z \ol{\pi}(a)$ with $|z| = m$ then $\sum_{w \in \bF_+^d} T_{z w, w} \otimes p_w = \ol{\pi}(a)$.
Moreover $T$ is a left lower triangular operator; indeed if $\nu \not\leq_l \mu$ then
\begin{align*}
\sca{L_z \ol{\pi}(a) \xi \otimes e_{\nu}, \eta \otimes e_\mu}
 =
\de_{z \nu, \mu} \sca{\al_{\ol{\nu}}(a)\xi, \eta}
 =
0.
\end{align*}
Hence $G_m(T) = \sum_{|\mu| = m} L_\mu \ol{\pi}(a_\mu)$ where $a_z = a$ and $a_\mu = 0$ for $\mu \neq z$.
Conversely suppose that $T$ satisfies these conditions.
Then for every finite subset $F_m$ of words of length $m$ we can verify that
\[
\| \sum_{\mu \in F_m} L_\mu \ol{\pi}(a_\mu) \| = \| \sum_{\mu \in F_m} L_\mu (L_\mu)^* G_m(T) \| \leq \nor{G_m(T)}
\]
since the $L_\mu (L_\mu)^*$ are pairwise orthogonal projections.
Therefore the net $(\sum_{\mu \in F_m} L_\mu \ol{\pi}(a_\mu))_{\{F_m:\textup{finite}\}}$ is bounded and thus the sum is the w*-limit of elements in $\scp{\A}{\al}{\L_d}$.
Hence every $G_m(T)$ is in $\scp{\A}{\al}{\L_d}$ and Fej\'{e}r's Lemma completes the proof.
\end{proof}

We turn our attention to dynamical systems $(\A, \{\al_i\}_{i \in [d]}$) where each $\al_i \in \End(\A)$ is induced by an invertible row operator $u_i$, i.e.
\begin{equation}\label{eq:uba defn}
\al_i(a) = \sum_{j_i \in [n_i]} u_{i, j_i} \, a \, v_{i, j_i} \foral a \in \A,
\end{equation}
where $v_i$ is the inverse of $u_i$.

\begin{definition}\label{D:uba defn}
We say that $\{\al_i\}_{i \in [d]}$ is a \emph{uniformly bounded spatial action} on a w*-closed algebra $\A$ of $\B(\H)$ if every $\al_i$ is implemented by an invertible row operator $u_i$ and $\{u_i\}_{i \in [d]}$ is uniformly bounded.
\end{definition}

\begin{proposition}\label{P:uba sys}
If $\{\al_i\}_{i \in [d]}$ is a uniformly bounded spatial action on a w*-closed algebra $\A$ of $\B(\H)$ then $(\A, \{\al_i\}_{i \in [d]})$ is a unital w*-dynamical system.
\end{proposition}

\begin{proof}
Let $\mu = \mu_m \dots \mu_1$ be a word in $\bF_+^d$.
Referring to Definition \ref{D:ubo defn} we verify that
\begin{align*}
\al_\mu(a)
& =
\al_{\mu_m} \cdots \al_{\mu_1}(a) \\
& =
\sum_{j_m \in [\mu_m]} \cdots \sum_{j_1 \in [\mu_1]}
u_{\mu_m, j_m} \cdots u_{\mu_1, j_1} a v_{\mu_1, j_1} \cdots v_{\mu_m, j_m} \\
& =
\wh{u}_{\mu_m \dots \mu_1} a \wh{v}_{\mu_1 \dots \mu_m}
\end{align*}
for all $a \in \A$.
Therefore $\nor{\al_\mu}_{cb} \leq \nor{\wh{u}_\mu} \cdot \nor{\wh{v}_\mu}$ so that $\al_\mu \in \End(\A)$.
As $\{u_i\}_{i \in [d]}$ and $\{v_i\}_{i \in [d]}$ are uniformly bounded by $K$ we derive that $\nor{\al_\mu} \leq K^2$ for all $\mu$, hence $\{\al_\mu\}_{\mu \in \bF_+^d}$ is uniformly bounded.
\end{proof}

The prototypical examples of uniformly bounded actions are systems implemented by Cuntz families.

\begin{examples}\label{E:odometer}
Every (unital) endomorphism of $\B(\H)$ is implemented by a countable Cuntz family when $\H$ is separable.
A proof can be found in \cite[Proposition 2.1]{Arv89}.
However the Cuntz family is not uniquely defined as shown by Laca \cite{Lac93}.

Examples of endomorphisms of maximal abelian selfadjoint algebras implemented by a Cuntz family have been considered by the second author and Peters \cite{KPe15}.
In particular let $\vphi \colon X \to X$ be an onto map on a measure space $(X, m)$ such that: (i) $\vphi$ and $\vphi^{-1}$ preserve the null sets; and (ii) there are $d$ Borel cross-sections $\psi_1, \dots, \psi_d$ of $\vphi$ with $\psi_i(X) \cap \psi_j(X) = \mt$ such that $\cup_{i=1}^d \psi_i(X)$ is almost equal to $X$.
Then it is shown in \cite[Proposition 2.2]{KPe15} that the endomorphism $\al\colon L^\infty(X) \to L^\infty(X)$ induced by $\vphi$ is realized through a Cuntz family.
Such cases arise in the context of $d$-to-$1$ local homeomorphisms for which an appropriate decomposition of $X$ into disjoint sets can be obtained \cite[Lemma 3.1]{KPe15}. 
As long as the boundaries of the components are null sets then the requirements of \cite[Proposition 2.2]{KPe15} are satisfied.
The prototypical example is the Cuntz-Krieger odometer, where
\[
X = \prod_k \{1, \dots, d\} \qand m = \prod_k m'
\]
for the averaging measure $m'$, and the backward shift $\vphi$ \cite[Example 3.3]{KPe15}.

The results of \cite{KPe15} follow the inspiring work of Courtney-Muhly-Schmidt \cite{CMS12} on endomorphisms $\al$ of the Hardy algebra induced by a Blaschke product $b$.
In particular it is shown in \cite[Corollary 3.5]{CMS12} that there is a Cuntz family implementing $\al$ if and only if there is a specific orthonormal basis $\{v_1, \dots, v_d\}$ for $H^2(\bT) \ominus b \cdot H^2(\bT)$.
An important part of the theory in \cite{CMS12} is the existence of a master isometry $C_b$, and the reformulation of the problem in terms of W*-correspondences when combined with \cite{Lac93}.
These elements pass on to the context of \cite{KPe15} where further necessary and sufficient conditions are given for a Cuntz family to implement an endomorphism of $L^\infty(X)$.
\end{examples}

Uniformly bounded actions extend to the entire $\B(\H)$ and we will use the same notation for their extensions.
By applying $u_{i, j_i}$ and $v_{i,j_i}$ on each side of equation (\ref{eq:uba defn}) we also get
\begin{equation}\label{eq:multiply}
\al_i(x) u_{i, j_i} = u_{i, j_i} x \qand v_{i, j_i} \al_i(x) = x v_{i, j_i}
\end{equation}
for every $x \in \B(\H)$.
The following proposition will be essential for our analysis of the bicommutant.

\begin{proposition}\label{P:com al}
Let $\al$ be an endomorphism of $\B(\H)$ induced by an invertible row operator $u = [u_i]_{i \in [n]}$ for some $n \in \bZ_+ \cup \{\infty\}$.
Then for any $x, y \in \B(\H)$ we have that
\[
\al(x) y = y \al(x) \qiff x \cdot v_j y u_k = v_j y u_k \cdot x \foral j, k \in [n]
\]
where $v = [v_i]_{i \in [n]}$ is the inverse of $u$.
\end{proposition}

\begin{proof}
Suppose first that $\al(x) y = y \al(x)$.
Then it follows that
\begin{align*}
x v_j y u_k
& =
v_j \al(x) y u_k
=
v_j y \al(x) u_k
=
v_j y u_k x
\end{align*}
for all $j, k \in [n]$.
Conversely if $x v_j y u_k = v_j y u_k x$ for all $j, k \in [n]$ then equation (\ref{eq:multiply}) yields
\begin{align*}
v_j \al(x) y u_k
& =
x v_j y u_k
=
v_j y u_k x
=
v_j y \al(x) u_k.
\end{align*}
Therefore we obtain
\begin{align*}
\al(x) y
& =
\sum_{j \in [n]} \sum_{k \in [n]} u_j (v_j \al(x) y u_k) v_k
=
\sum_{j \in [n]} \sum_{k \in [n]} u_j (v_j y \al(x) u_k) v_k
=
y \al(x)
\end{align*}
and the proof is complete.
\end{proof}

\begin{remark}
If $\al \in \End(\A)$ is induced by an invertible row operator $u$ then $\al$ extends to an endomorphism of $\A''$.
Indeed by Proposition \ref{P:com al} we have that $v_j y u_k \in \A'$ for all $y \in \A'$ since $\A' \subseteq \al(\A)'$.
Hence if $z \in \A''$ then $z v_j y u_k = v_j y u_k z$ for all $y \in \A'$.
Applying Proposition \ref{P:com al} again yields $\al(z) \in \A''$.

Therefore given a w*-dynamical system $(\A, \{\al_i\}_{i \in [d]})$ where each $\al_i$ is implemented by an invertible row operator $u_i$ then we automatically have the induced systems $(\B(\H), \{\al_i\}_{i \in [d]})$ and $(\A'', \{\al_i\}_{i \in [d]})$.
Hence the w*-semicrossed products
\[
\scp{\A}{\al}{\L_d} \, , \, \, \scp{\A}{\al}{\R_d} \, , \, \, \scp{\B(\H)}{\al}{\L_d} \, , \, \, \scp{\B(\H)}{\al}{\R_d} \, , \, \, \scp{\A''}{\al}{\L_d} \, , \, \, \scp{\A''}{\al}{\R_d}
\]
are all well defined.
\end{remark}

There are also two more algebras linked to our analysis.
Suppose that $\{\al_i\}_{i \in [d]}$ are endomorphisms of $\B(\H)$ and each $\al_i$ is induced by an invertible row operator $u_i$.
Then we can form the free semigroup $\bF_+^N$ for $N = n_1 + \cdots + n_d$.
Since we want to keep track of the generators we write
\[
\bF_+^N = \sca{(i, j) \mid i \in [d], j \in [n_i]} = \ast_{i \in [d]} \bF_+^{n_i}.
\]
We fix the operators
\[
V_{i,j} = u_{i,j} \otimes \Bl_i \qand W_{i,j} = u_{i,j} \otimes \Br_i \; \foral \; (i,j) \in ([d], [n_i])
\]
and the representation $\rho \colon \B(\H) \to \B(\H \otimes \ell^2(\bF_+^d))$ with $\rho(x) = x \otimes I$.

\begin{definition}
With the aforementioned notation, we define the spaces
\[
\scp{\A'}{u}{\L_d} := \ol{\spn}^{\textup{w*}}\{V_{i,j} \rho(y) \mid (i,j) \in ([d], [n_i]), y \in \A' \}
\]
and
\[
\scp{\A'}{u}{\R_d} := \ol{\spn}^{\textup{w*}}\{W_{i,j} \rho(y) \mid (i,j) \in ([d], [n_i]), y \in \A' \}.
\]
\end{definition}

Notice here that for a word $\bo{w} = (\mu_k, j_{\mu_k}) \dots (\mu_1, j_{\mu_1}) \in \bF_+^N$ we have
\begin{align*}
V_{\bo{w}}
& = 
L_{\mu_k} \rho(u_{\mu_k, j_{\mu_k}}) \cdots L_{\mu_1} \rho(u_{\mu_1, j_{\mu_1}}) 
 =
L_{\mu_k \dots \mu_1} \rho(u_{\bo{w}}).
\end{align*}
The generators satisfy a set of covariance relations which we will use to show that the above spaces are algebras.

\begin{proposition}\label{P:com is alg}
Let $(\A, \{\al_i\}_{i \in [d]})$ be a w*-dynamical system such that each $\al_i$ is implemented by an invertible row operator $u_i$.
Then
\[
\scp{\A'}{u}{\L_d} = \ol{\alg}^{\textup{w*}} \{ V_{\bo{w}} \rho(y) \mid \bo{w} \in \bF_+^N, y \in \A'\}
\]
and
\[
\scp{\A'}{u}{\R_d} = \ol{\alg}^{\textup{w*}} \{ W_{\bo{w}} \rho(y) \mid \bo{w} \in \bF_+^N, y \in \A'\}
\]
where $\bF_+^N = \sca{(i,j) \mid i \in [d], j \in [n_i]}$.
\end{proposition}

\begin{proof}
We prove the left version.
The right version follows by similar arguments.
It suffices to show that $\rho(y) L_i \rho(u_{i,j})$ is in $\scp{\A'}{u}{\L_d}$ for all $y \in \A'$ and $(i,j) \in ([d], [n_i])$.
Suppose that $v_i = [v_{i,j_i}]_{j_i \in [n_i]}$ is the inverse of $u_i$.
Then we can write
\[
y 
= 
\sum_{k \in [n_i]} \sum_{l \in [n_i]} u_{i, k} v_{i, k} y u_{i, l} v_{i, l}
=
\sum_{k \in [n_i]} \sum_{l \in [n_i]} u_{i, k} y_{i, k, l} v_{i, l}
\]
where $y_{i,k, l} := v_{i, k} y u_{i, l}$.
Proposition \ref{P:com al} yields that $y_{i,k, l}$ is in $\A'$ since $y \in \A' \subseteq \al_i(\A)'$.
Therefore we have that
\begin{align*}
y u_{i,j}
& =
\sum_{k \in [n_i]} \sum_{l \in [n_i]} u_{i, k} y_{i,k, l} v_{i, l} u_{i,j}
 =
\sum_{k\in [n_i]} u_{i, k} y_{i, k, j}
\end{align*}
which gives that
\begin{align*}
\rho(y) L_i \rho(u_{i,j})
& = 
L_i \rho(y) \rho(u_{i,j})
=
\sum_{k\in [n_i]} L_i \rho(u_{i, k} y_{i, k, j})
=
\sum_{k\in [n_i]} V_{i,k} \rho(y_{i, k, j}).
\end{align*}
Recall that $\| \sum_{k \in F} u_{i, k} v_{i, k} \| \leq 1$ for every finite subset $F$ of $[n_i]$, hence
\[
\| \sum_{k\in F} u_{i, k} y_{i, k, j} \|
=
\| \sum_{k\in F} u_{i, k} v_{i, k} y u_{i, j} \|
\leq 
\nor{y} \nor{u_{i, j}}.
\]
Thus the net $\left( \sum_{k \in F} u_{i, k} y_{i, k, j} \right)_{\{F:\textup{finite}\}}$ is bounded and the sum above converges in the w*-topology.
Hence the element $\rho(y) L_i \rho(u_{i,j})$ is in $\scp{\A'}{u}{\L_d}$.
\end{proof}

\subsection{Dynamical systems over $\bZ_+^d$}\label{Ss:com}

Similarly we define a (unital) w*-dynamical system $(A,\al,\bZ_+^d)$ to consist of a semigroup action $\al\colon \bZ_+^d \to \End(\A)$ such that
\[
\sup \{ \nor{\al_{\un{n}}} \mid \un{n} \in \bZ_+^d\} < \infty.
\]
Since the action is generated by $d$ commuting endomorphisms $\al_\Bi$ it suffices to have that $\sup\{ \nor{\al_\Bi^n} \mid n \in \bZ_+\} < \infty$ for all $i \in [d]$.
Consequently commuting spatial actions $\al_{\Bi}$ that are uniformly bounded in the sense of Definition \ref{D:uba defn} induce unital w*-dynamical systems.

Examples are given by actions implemented by a unitarizable semigroup homomorphism of $\bZ_+^d$ in $\B(\H)$.
However our setting accommodates cases where each $\al_{\Bi}$ may be implemented by an invertible element \emph{separately}.
This gives us the opportunity to tackle more commuting actions.
Let us illustrate this with an example.

\begin{example}\label{E:Weyl}
Every pair of unitaries $U, V$ that satisfy Weyl's relation $UV = \la VU$ for $\la \in \bT$ obviously implements two commuting actions $\al_{\bo{1}} = \ad_U$ and $\al_{\bo{2}} = \ad_V$ on $\B(\H)$.
In fact it is not difficult to show that every action $\al \colon \bZ_+^2 \to \Aut(\B(\H))$ is indeed of this form: $\al_{\bo{1}}$ and $\al_{\bo{2}}$ will be implemented by unitaries that commute modulo a $\la \in \bT$.
This follows in the same way as in \cite[Theorem 9.3.3]{KR86}.
\end{example}

\begin{remark}\label{R: Laca}
Results of Laca \cite{Lac93} give a general criterion for commuting normal $*$-endomorphisms of $\B(\H)$.
Suppose that $\al, \be \in \End(\B(\H))$ commute and are given by
\[
\al(x) = \sum_{i \in [n]} s_i x s_i^* \qand \be(x) = \sum_{j \in [m]} t_j x t_j^* 
\]
for the Cuntz families $\{s_i\}_{i \in [n]}$ and $\{t_j\}_{j \in [m]}$.
Therefore
\[
\sum_{i \in [n]} \sum_{j \in [m]} s_i t_j x t_j^* s_i^* = \sum_{j \in [m]} \sum_{i \in [n]} t_j s_i x s_i^* t_j^*.
\]
Notice that on each side we sum up orthogonal representations of $\B(\H)$ and thus we can take the limits so that
\[
\sum_{(i,j) \in [n] \times [m]} s_i t_j x t_j^* s_i^* = 
\sum_{(i,j) \in [n] \times [m]} t_j s_i x s_i^* t_j^*.
\]
We may see the families $\{s_i t_j\}_{(i,j) \in [n] \times [m]}$ and $\{t_j s_i\}_{(i,j) \times [n] \times [m]}$ as representations of the Cuntz algebra $\O_{n \cdot m}$.
Applying \cite[Proposition 2.2]{Lac93} gives a unitary operator $W = [w_{(k,l), (i,j)}]$ in $\M_{nm}(\bC)$ such that
\begin{align*}
t_j s_i 
 = 
\sum_{(k,l) \in [n] \times [m]} w_{(k,l), (i,j)} s_k t_l.
\end{align*}

This criterion can be used to research the class of endomorphisms $\al$ that commute with a fixed $\be$.
We show how this can be done in the next two examples.
\end{remark}

\begin{example}\label{E:com Cuntz 1}
For this example fix $\H = \ell^2(\bZ_+)$ and let the Cuntz family
\[
S_1 e_n = e_{2n} \qand S_2 e_n = e_{2n+1}.
\]
Let $U \in \B(\H)$ be a unitary and fix the induced actions
\[
\al(x) = U x U^* \qand \be(x) = S_1 x S_1^* + S_2 x S_2^*.
\]
We will show that $\al$ and $\be$ commute if and only if
\begin{equation}\label{eq:u form 1}
U = \la \diag\{\mu^{\phi(n)} \mid n \in \bZ_+\} \qfor \la, \mu \in \bT,
\end{equation}
where $\phi(n)$ is the sequence of the binary weights of $n$, i.e.
\[
\phi(n) = \# \text{ of $1$'s appearing in the binary expansion of $n$}.
\]

First suppose that $\al$ commutes with $\be$. 
By Remark \ref{R: Laca} there exists a unitary
\[
W = \begin{bmatrix} a & c \\ b & d \end{bmatrix} \in \M_2(\bC)
\]
such that
\[
US_1 = a S_1 U + b S_2 U \qand US_2 = c S_1U + d S_2 U.
\]
Below we write 
\[
U e_k = \sum_n \la^{(k)}_n e_n \foral k \in \bZ_+.
\]
Since $S_1 e_0 = e_0$ we have
\begin{align*}
\sum_n \la_n^{(0)} e_n 
& = U e_0 = U S_1 e_0 \\
& = a S_1 U e_0 + b S_2 U e_0 \\
& =  \sum_n a \la_n^{(0)} e_{2n} + b \la_n^{(0)} e_{2n+ 1}.
\end{align*}
We thus obtain
\begin{equation}\label{eq:la soln 1}
\la_0^{(0)} = a \la_0^{(0)} \qand 
\la_{2n}^{(0)} = a \la_n^{(0)}, 
\la_{2n+1}^{(0)} = b \la_{n}^{(0)}
\foral n \geq 1.
\end{equation}
Therefore if $\la_0^{(0)} = 0$ then $Ue_0 = 0$ which is a contradiction to $U$ being a unitary.
Hence $a = 1$ from the first equation and thus $b = c = 0$ and $|d| = 1$, since $W$ is a unitary.
Thus we obtain
\[
US_1 = S_1 U \qand US_2 = d S_2 U.
\]
Consequently we get
\[
U = US_1S_1^* + US_2 S_2^* = S_1 U S_1^* + d S_2 U S_2^*.
\]
In addition, applying $b=0$ in equality (\ref{eq:la soln 1}) gives that
\[
\begin{cases}
\la^{(0)}_1 = b \la^{(0)}_0 = 0, \\
\la_2^{(0)} = a \la_1^{(0)} = 0, \\
\la_3^{(0)} = b \la_2^{(0)} = 0, \\
\la_4^{(0)} = a \la_2^{(0)} = 0, \\
\phantom{o} \vdots
\end{cases}
\]
and inductively we have that $\la_n^{(0)} = 0$ for all $n \geq 1$.
Hence $Ue_0 = \la_0^{(0)} e_0$.
In particular we get that $|\la_0^{(0)}| = 1$ and therefore
\begin{align*}
U = \begin{bmatrix} \la_0^{(0)} & 0 \\ 0 & \ast \end{bmatrix}
\end{align*}
when decomposing $\H = \sca{e_0} \oplus \sca{e_0}^{\perp}$.
Now we apply for $e_1$ to obtain
\begin{align*}
U e_1 = d S_2 U S_2^* e_1 = d S_2 U e_0 = \la_0^{(0)} d e_1
\end{align*}
from which we get
\begin{align*}
\la_{1}^{(1)} = \la_0^{(0)} d \qand
\la_{n}^{(1)} = 0 \text{ for $n \neq 1$}.
\end{align*}
As $\la_{1}^{(1)}$ has modulus $1$ we then get that
\[
U =
\begin{bmatrix}
\la_0^{(0)} & 0 & 0 \\ 0 & \la_0^{(0)} d & 0 \\ 0 & 0 & \ast 
\end{bmatrix}
\]
Now applying for $e_2$ we get
\[
U e_2 = S_1 U S_1^* e_2 = S_1 U e_1 = \la_0^{(0)} d e_2
\]
and therefore
\[
U =
\begin{bmatrix}
\la_0^{(0)} & 0 & 0 & 0 \\ 0 & \la_0^{(0)} d & 0 & 0 \\ 0 & 0 & \la_0^{(0)} d & 0 \\ 0 & 0 & 0 & \ast 
\end{bmatrix}.
\]
Hence we have verified equation (\ref{eq:u form 1}) for $n = 0, 1, 2$ with 
\[\la = \la^{(0)}_0 \qand \mu = d.
\]
Now suppose that $U e_n = \la \mu^{\phi(n)} e_n$ holds for every $n < 2k$ with $k \neq 0$; then
\[
U e_{2k} = S_1 U S_1^* e_{2k} = S_1 U e_k = \la \mu^{\phi(k)} e_{2k}
\]
as $\phi(2k) = \phi(k)$.
On the other hand if $U e_n = \la \mu^{\phi(n)} e_n$ holds for every $n < 2k + 1$ then
\[
U e_{2k + 1} = \mu S_2 U S_2^* e_{2k +1} = \mu S_2 U e_{k} = \la \mu^{\phi(k) + 1} e_{2k + 1}
\]
since
\[
\phi(2k+1) = \phi(2k) + 1 = \phi(k) + 1.
\]
By using strong induction we have that $U$ satisfies equation (\ref{eq:u form 1}).

Conversely suppose that $U$ is as in equation (\ref{eq:u form 1}).
We will show that the induced actions $\al$ and $\be$ commute.
First we consider $x = e_i \otimes e_j^*$, the rank one operator sending $e_j$ to $e_i$.
A direct computation shows that
\begin{align*}
\al \be (x) e_n
& =
\begin{cases}
d^{\phi(2i) - \phi(2k)} e_{2i} \sca{e_k, e_j} & \text{ if } n = 2k, \\
d^{\phi(2i + 1) - \phi(2k +1)} e_{2i + 1} \sca{e_k, e_j} & \text{ if } n = 2k+1.
\end{cases}
\end{align*}
On the other hand we have that
\begin{align*}
\be \al(x) e_n
& = 
\begin{cases}
d^{\phi(i) - \phi(k)} e_{2i} \sca{e_k, e_j} & \text{ if } n =2k,\\
d^{\phi(i) - \phi(k)} e_{2i + 1} \sca{e_k, e_j} & \text{ if } n=2k+1.
\end{cases}
\end{align*}
Since 
\[
\phi(2k) - \phi(2i) = \phi(k) - \phi(i)
\]
and
\[
\phi(2k+1) - \phi(2i+1) = \phi(2k) + 1 - \phi(2i) - 1 = \phi(k) - \phi(i)
\]
we obtain that $\al \be(x) = \be \al(x)$.
Since $\al, \be$ are sot-continuous (being implemented by operators), passing to sot-limits yields that $\al$ and $\be$ commute.
\end{example}

\begin{example}\label{E:com Cuntz 2}
For this example we let $\H = \ell^2(\bZ)$ and the Cuntz family
\[
S_1 e_n = e_{2n} \qand S_2 e_n = e_{2n+1}.
\]
Let $U \in \B(\H)$ be a unitary and write $\ell^2(\bZ) = H_1 \oplus H_2$ for
\[
H_1 = \sca{e_n \mid n \geq 0} \qand H_2 = \sca{e_n \mid n \leq -1}.
\]
We claim that the actions induced by $U$ and $\{S_1, S_2\}$ commute if and only if $U$ attains one of the forms
\begin{equation}\label{eq:u form 2}
U = \la I_{H_1} \oplus \mu I_{H_2} 
\; \text{ or } \;
U =
\begin{bmatrix} 0 & \mu w^* \\ \la w & 0 \end{bmatrix}
\end{equation}
where $\la, \mu \in \bT$ and $w \in \B(H_1, H_2)$ is the unitary with $w e_{n} = e_{-n-1}$.

If the actions commute then by Remark \ref{R: Laca} there exists a unitary
\[
W = \begin{bmatrix} a & c \\ b & d \end{bmatrix} \in \M_2(\bC)
\]
such that
\[
US_1 = a S_1 U + b S_2 U \qand US_2 = c S_1U + d S_2 U.
\]
Below we write 
\[
U e_k = \sum_n \la^{(k)}_n e_n \foral k \in \bZ.
\]
Since $S_1 e_0 = e_0$ we obtain
\begin{align*}
\sum_n \la_n^{(0)} e_n
& =
U e_0
=
U S_1 e_0 \\
& =
(a S_1 + b S_2) U e_0 \\
& =
\sum_n a \la_n^{(0)} e_{2n} + b \la_n^{(0)} e_{2n +1}.
\end{align*}
Consequently
\[
\la_{2k}^{(0)} = a \la_k^{(0)} \qand 
\la_{2k + 1}^{(0)} = b \la_k^{(0)} \;
\foral k \in \bZ.
\]

If $a =1$ then $b=0$ as $|a|^2 + |b|^2 = 1$.
Now, if $a \neq 1$ then $\la_0^{(0)} = 0$ and thus $\la_n^{(0)} = 0$ for all $n \geq 0$.
If, in addition, $a \neq 0$ then also $b \neq 1$ and so $\la_{-1}^{(0)} = 0$ which implies that $\la_n^{(0)} = 0$ for all $n \leq 0$.
This contradicts that $U$ is a unitary.
Therefore if $a \neq 1$ then it must be that $a = 0$ in which case we get that $|b| = 1$.
However a symmetrical argument shows that if $a =0$ and $b \neq 1$ then $U e_0=0$ which is a contradiction.
Therefore if $a \neq 1$ then $a = 0$ and $b = 1$.
Consequently we have the following cases:
\[
\textup{(i)} \; a = 1, b = 0 \quad \textup{or} \quad \textup{(ii)} \; a = 0, b =1.
\]

\medskip

\noindent $\bullet$ Case (i).
When $a=1$ and $b=0$ then $c=0$ and $d \in \bT$ and therefore
\[
US_1 = S_1 U \qand US_2 = d S_2 U
\]
which we can rewrite as
\[
U = S_1 U S_1^* + d S_2 U S_2^*.
\]
Applying for $e_{-1}$ we obtain
\[
\sum_n \la_n^{(-1)} e_n = U e_{-1} = d S_2 U S_2^* e_{-1} = \sum_n d \la_n^{(-1)} e_{2n+1}.
\]
Hence we get that
\[
\begin{cases}
\la_0^{(-1)} = 0\\
\la_1^{(-1)} = d \la_0^{(-1)} = 0\\
\la_2^{(-1)} = 0 \\
\la_3^{(-1)} = d \la_1^{(-1)} = 0 \\
\phantom{o} \vdots
\end{cases}
\qand
\begin{cases}
\la_{-1}^{(-1)} = d \la_{-1}^{(-1)}\\
\la_{-2}^{(-1)} = 0\\
\la_{-3}^{(-1)} = d \la_{-1}^{(-1)} \\
\la_{-4}^{(-1)} = 0 \\
\phantom{oo} \vdots
\end{cases}
\]
It follows that $d=1$ otherwise $Ue_{-1} = 0$ which is a contradiction.
Therefore we derive that
\[
U = S_1 U S_1^* + S_2 U S_2^*.
\]
Hence we have that $U e_0 = \la e_0$ for $\la = \la_0^{(0)}$ and so $U e_n = \la e_n$ when $n \geq 0$ as in Example \ref{E:com Cuntz 1}.
On the other hand $U e_{-1} = \mu e_{-1}$ for $\mu = \la_{-1}^{(-1)}$ and so $U e_n = \mu e_n$ when $n < 0$ by similar computations.
Thus it follows that
\[
U = \la I_{H_1} \oplus \mu I_{H_2} \qfor \la, \mu \in \bT.
\]

\medskip

\noindent $\bullet$ Case (ii).
When $a = 0$ and $b=1$ then $c \in \bT$ and $d=0$ in which case we have
\[
U S_1 = S_2 U \qand U S_2 = c S_1 U
\]
or equivalently
\[
U = S_2 U S_1^* + c S_1 U S_2^*.
\]
By applying on $e_{-1}$ we get
\[
\begin{cases}
\la_0^{(-1)} = c \la_0^{(-1)}, \\
\la_1^{(-1)} = \la_3^{(-1)} = \cdots = 0, \\
\la_2^{(-1)} = c \la_{1}^{(-1)} = 0, \\
\la_4^{(-1)} = \la_6^{(-1)} = \cdots  = 0,
\end{cases}
\qand
\begin{cases}
\la_{-1}^{(-1)} = \la_{-3}^{(-1)} = \cdots = 0, \\
\la_{-2}^{(-1)} = c \la_{-1}^{(-1)} = 0, \\
\la_{-4}^{(-1)} = \la_{-6}^{(-1)} = \cdots  = 0.
\end{cases}
\]
If $c \neq 1$ then we would get that $U e_{-1} = 0$ which is a contradiction.
Therefore we obtain that $c=1$ and thus
\begin{equation}\label{eq:u 2}
U = S_2 U S_1^* + S_1 U S_2^*.
\end{equation}
In this case we have that
\[
U e_0 = \la e_{-1} \qand U e_{-1} = \mu e_{0}
\]
for $\la, \mu \in \bT$.
We claim that
\[
U =
\begin{bmatrix} 0 & \mu w^* \\ \la w & 0 \end{bmatrix}
\]
for $\ell^2(\bZ) = H_1 \oplus H_2$ and the unitary $w \in \B(H_1, H_2)$ with $w e_{n} = e_{-n-1}$, i.e.
\[
Ue_n =
\begin{cases}
\la e_{-n - 1} & \text{ if } n \geq 0, \\
\mu e_{-n-1} & \text{ if } n \leq -1.
\end{cases}
\]
Indeed this holds for $n=0, -1$.
Let $n \geq 0$ and suppose it holds for every $0 \leq k < n$.
If $n = 2k$ then by the inductive hypothesis and equation (\ref{eq:u 2}) we get
\begin{align*}
U e_{n} 
& = S_2 U S_1^* e_{2k} = S_2 U e_k 
 = \la S_2 e_{-k -1} = \la e_{-2k - 1} = \la e_{- n - 1} 
\end{align*}
whereas if $n = 2k +1$ we get
\begin{align*}
U e_n
& = S_1 U S_2^* e_{2k +1} = S_1 U e_k 
 = \la S_1 e_{-k-1} = \la e_{-2k -2} = \la e_{-n - 1}.
\end{align*}
A similar computation holds for $n \leq -1$.
Strong induction then completes the proof of the claim.

Conversely if a unitary $U$ satisfies equation (\ref{eq:u form 2}) then $\ad_U$ either fixes or interchanges $S_1$ and $S_2$.
In either case we get
\[
US_1U^* y US_1^* U^* + US_2U^* y US_2^* U^* = S_1 y S_1^* + S_2 y S_2^*
\]
for all $y \in \B(\H)$.
Applying for $y = U x U^*$ yields that the actions induced by $U$ and $\{S_1, S_2\}$ commute.
\end{example}

Now we return to the definition of the semicrossed product for actions of $\bZ_+^d$.
On $\H \otimes \ell^2(\bZ_+^d)$ we define the representation $\pi \colon \A \to \B(\H \otimes \ell^2(\bZ_+^d))$ and the creation operators $L\colon \bZ_+^d \to \B(\H \otimes \ell^2(\bZ_+^d))$ by
\[
\pi(a) \xi \otimes e_{\un{n}} = \al_{\un{n}}(a)\xi \otimes e_{\un{n}}
\qand
L_{\Bi}\xi \otimes e_{\un{n}} = \xi \otimes e_{\Bi + \un{n}}.
\]
Notice here that due to commutativity of $\bZ_+^d$ we make no distinction between right and left versions.

\begin{definition}
Let $(\A,\al,\bZ_+^d)$ be a unital w*-dynamical system.
We define the w*-semicrossed product
\[
\scp{\A}{\al}{\bZ_+^d}:= \ol{\spn}^{\textup{w*}}\{L_{\un{n}} \pi(a) \mid a \in \A, \un{n} \in \bZ_+^d\}.
\]
\end{definition}

Again we can directly verify the covariance relations by applying on the elementary tensors.
In analogy to Proposition \ref{P:Fc al} we have the following proposition.
For its proof we may again invoke a Fej\'{e}r-type argument for the appropriate Fourier co-efficients induced by $\{U_{\un{s}}\}$ with $\un{s} \in [-\pi, \pi]^d$.

\begin{proposition}\label{P:Fc al com}
Let $(\A,\al,\bZ_+^d)$ be a unital w*-dynamical system.
Then an operator $T \in \B(\H \otimes \ell^2(\bZ_+^d))$ is in $\scp{\A}{\al}{\bZ_+^d}$ if and only if it is lower triangular and 
\[
G_{\un{m}}(T) 
= 
L_{\un{m}} \pi(a_{\un{m}}) \qfor a_{\un{m}} \in \A
\]
for all $\un{m} \in \bZ_+^d$.
\end{proposition}

Moreover we can proceed to a decomposition into subsequent one-dimen\-sional w*-semicrossed products.

\begin{proposition}\label{P:disintegrate}
Let $(\A, \al, \bZ_+^d)$ be a unital w*-dynamical system.
Then $\scp{\A}{\al}{\bZ_+^d}$ is unitarily equivalent to
\[
\left( \cdots \left( ( \scp{\A}{\al_{\bo{1}}}{\bZ_+} ) \ol{\times}_{\wh{\al}_{\bo{2}}} \bZ_+ \right) \cdots \right) \ol{\times}_{\wh{\al}_{\bo{d}}} \bZ_+
\]
where $\wh{\al}_{\bo{i}} = \al_{\bo{i}} \otimes^{(i-1)} \id$ for $i = 2, \dots, d$.
\end{proposition}

\begin{proof}
We show how this decomposition works when $d=2$; the general case follows by iterating.
Fix $\al_{\bo{1}}$ and $\al_{\bo{2}}$ commuting endomorphisms of $\A$.
Then $\scp{\A}{\al_{\bo{1}}}{\bZ_+}$ acts on $\H \otimes \ell^2$ by 
\[
\pi(a) \xi \otimes e_n = \al_{(n,0)}(a)\xi \otimes e_n
\qand
L_1 \xi \otimes e_n = \xi \otimes e_{n+1}.
\]
Now we define the w*-dynamical system $(\scp{\A}{\al_{\bo{1}}}{\bZ_+}, \wh{\al}_{\bo{2}}, \bZ_+)$ by setting
\[
\wh{\al}_{\bo{2}}(\pi(a)) = \pi\al_{\bo{2}}(a) \qand \wh{\al}_{\bo{2}}(L_1) = L_1.
\]
To see that $\wh{\al}_{\bo{2}}$ defines a w*-continuous completely bounded endomorphism on $\scp{\A}{\al_{\bo{1}}}{\bZ_+}$ first note that $\scp{\A}{\al_{\bo{1}}}{\bZ_+}$ is a w*-closed subalgebra of $\A \, \ol{\otimes} \, \B(\ell^2)$.
Since $\al_{\bo{2}}$ is w*-continuous and completely bounded, for $X \in \A \, \ol{\otimes} \, \B(\ell^2)$ we can obtain $\al_{\bo{2}} \otimes \id(X)$ as the limit of 
\[
\al_{\bo{2}} \otimes \id_n(P_{\H \otimes \ell^2(n)} X |_{\H \otimes \ell^2(n)}) \in \A \otimes \M_n(\bC).
\]
Hence $\al_{\bo{2}} \otimes \id$ defines a w*-completely bounded endomorphism of $\A \, \ol{\otimes} \, \B(\ell^2)$ and $\wh{\al}_{\bo{2}}$ is its restriction to the  $\scp{\A}{\al_{\bo{1}}}{\bZ_+}$.
The unitary $U$ given by $U \xi \otimes e_{(n,m)} = \xi \otimes e_n \otimes e_m$ then defines the required unitary equivalence between $\scp{\A}{\al}{\bZ_+^2}$ and $\scp{(\scp{\A}{\al_{\bo{1}}}{\bZ_+})}{\wh{\al}_{\bo{2}}}{\bZ_+}$.
\end{proof}

\section{The bicommutant property}\label{S:commutant}

\subsection{Semicrossed products over $\bF_+^d$}\label{Ss:scp free bicom}

The duality between the left and the right w*-semicrossed products is reflected in the bicommutant property.

\begin{theorem}\label{T:com cun}
Let $(\A, \{\al_i\}_{i \in [d]})$ be a w*-dynamical system of a uniformly bounded spatial action implemented by $\{u_i\}_{i \in [d]}$.
Then we have that
\[
(\scp{\A}{\al}{\L_d})'  = \scp{\A'}{u}{\R_d} \qand (\scp{\A'}{u}{\L_d})' = \scp{\A''}{\al}{\R_d}
\]
and that
\[
(\scp{\A}{\al}{\R_d})'  = \scp{\A'}{u}{\L_d} \qand (\scp{\A'}{u}{\R_d})' = \scp{\A''}{\al}{\L_d}.
\]
\end{theorem}

\begin{proof}
Direct computations show that $\scp{\A'}{u}{\R_d}$ is in the commutant of $\scp{\A}{\al}{\L_d}$.
For the reverse inclusion let $T$ be in the commutant of $\scp{\A}{\al}{\L_d}$.
As the Fourier transform respects the commutant it suffices to show that $G_m(T)$ is in $\scp{\A'}{u}{\R_d}$ for all $m \in \bZ_+$, and it is zero for all $m < 0$.

For $\mu, \nu \in \bF_+^d$ and by using the commutant property we get that
\begin{align*}
\sca{T_{\mu, \nu} \xi, \eta}
& = 
\sca{T L_\nu \xi \otimes e_{\mt}, \eta \otimes e_\mu} \\
& =
\sca{L_\nu T \xi \otimes e_{\mt}, \eta \otimes e_\mu} 
=
\sca{T \xi \otimes e_{\mt}, \eta \otimes \Bl_\nu^*e_\mu}.
\end{align*}
However we have that $(\Bl_\nu)^* e_\mu = 0$ whenever $\nu \not\leq_r \mu$.
Therefore $T$ is right lower triangular and thus
\[
G_m(T)
=
\begin{cases}
\sum_{|\mu| = m} R_{\mu} T_{(\mu)} & \text{ if } m \geq 0, \\
0 & \text{ if } m < 0,
\end{cases}
\]
for $T_{(\mu)} = \sum_{w \in \bF_+^d} T_{w \ol{\mu}, w} \otimes p_w = R_\mu^* G_m(T)$.
Moreover we have that
\begin{align*}
\sum_{|\mu| = m} T_{w \ol{\mu}, w} \xi \otimes e_{w \ol{\mu}}
& =
G_m(T) L_w \xi \otimes e_\mt \\
& =
L_w G_m(T) \xi \otimes e_\mt
=
\sum_{|\mu| = m} T_{\ol{\mu}, \mt} \xi \otimes e_{w \ol{\mu}}
\end{align*}
which shows that $T_{(\mu)} = \rho(T_{\ol{\mu}, \mt})$ for all $\mu$ of length $m$.
Furthermore we have that
\begin{align*}
\sum_{|\mu| = m} T_{\ol{\mu}, \mt} a \xi \otimes e_{\ol{\mu}}
& =
G_m(T) \ol{\pi}(a) \xi \otimes e_\mt \\
& =
\ol{\pi}(a) G_m(T) \xi \otimes e_\mt
=
\sum_{|\mu| = m} \al_{\mu}(a)T_{\ol{\mu}, \mt} \xi \otimes e_{\ol{\mu}} 
\end{align*}
and therefore $T_{\ol{\mu}, \mt} a = \al_{\mu}(a) T_{\ol{\mu}, \mt}$ for all $a \in \A$.
Let $v_i$ be the inverse of $u_i$.
For $\mu = \mu_m \dots \mu_1$ and $j_i \in [n_{\mu_i}]$ we set
\[
y_{\mu, j_1, \dots, j_m} : = v_{\mu_1, j_1} \cdots v_{\mu_m, j_m} T_{\ol{\mu}, \mt}.
\]
Then $y_{\mu, j_1, \dots, j_m}$ is in $\A'$ since
\begin{align*}
a \cdot v_{\mu_1, j_1} \cdots v_{\mu_m, j_m} T_{\ol{\mu}, \mt}
& =
v_{\mu_1, j_1} \cdots v_{\mu_m, j_m} \al_{\mu_m} \cdots \al_{\mu_1}(a) T_{\ol{\mu}, \mt} \\
& =
v_{\mu_1, j_1} \cdots v_{\mu_m, j_m} \al_{\mu}(a) T_{\ol{\mu}, \mt} \\
& =
v_{\mu_1, j_1} \cdots v_{\mu_m, j_m} T_{\ol{\mu}, \mt} \cdot a
\end{align*}
for all $a \in \A$.
Now we can write
\begin{align*}
R_\mu T_{(\mu)}
& =
\sum_{j_m \in [n_{\mu_m}]} \cdots \sum_{j_1 \in [n_{\mu_1}]} R_{\mu} \rho(u_{\mu_m, j_m} \cdots u_{\mu_1, j_1}) \rho(y_{\mu, j_1, \dots, j_m}) \\
& =
\sum_{j_m \in [n_{\mu_m}]} \cdots \sum_{j_1 \in [n_{\mu_1}]} W_{\mu_m, j_m} \cdots W_{\mu_1, j_1} \rho(y_{\mu, j_1, \dots, j_m}).
\end{align*}
If $F$ is a finite set of $[n_{\mu_m}]$ then
\begin{align*}
\| \sum_{j_1 \in F} W_{\mu_m, j_m} \cdots W_{\mu_1, j_1} \rho(y_{\mu, j_1, \dots, j_m}) \|
& = \\
& \hspace{-5cm} =
\| \sum_{j_1 \in F} u_{\mu_m, j_m} \cdots u_{\mu_1, j_1} v_{\mu_1, j_1} \cdots v_{\mu_m, j_m} T_{\ol{\mu}, \mt} \| \\
& \hspace{-5cm} \leq
\nor{u_{\mu_m, j_m} \cdots u_{\mu_2, j_2}} \| \sum_{j_1 \in F} u_{\mu_1, j_1} v_{\mu_1, j_1} \| \nor{v_{\mu_2, j_2} \cdots v_{\mu_m, j_m}} \nor{T_{\ol{\mu}, \mt}} \\
& \hspace{-5cm} \leq
K^2 \nor{T_{\ol{\mu}, \mt}}
\end{align*}
where $K$ is the uniform bound for $\{\wh{u}_{\mu}\}_{\mu}$ and $\{\wh{v}_{\mu}\}_{\mu}$.
Inductively we have that the sums in the above form of $R_\mu T_{(\mu)}$ converge in the w*-topology and therefore each $R_{\mu} T_{(\mu)}$ is in $\scp{\A'}{u}{\R_d}$.
As in Proposition \ref{P:lower triangular} an application of Fej\'{e}r's Lemma induces that $T$ is in $\scp{\A'}{u}{\R_d}$.

Next we show that $(\scp{\A'}{u}{\L_d})' = \scp{\A''}{\al}{\R_d}$.
Again it is immediate that $\scp{\A''}{\al}{\R_d}$ is in the commutant of $\scp{\A'}{u}{\L_d}$.
For the reverse inclusion let $T$ be in the commutant.
Then $T$ commutes with all $L_{i} \rho(u_{i, j_i})$.
First let $\nu \not\leq_r \mu$ with $\nu = \nu_k \dots \nu_1$; then
\begin{align*}
\sca{T_{\mu,\nu} u_{\nu_k, j_{k}} \dots u_{\nu_1, j_1} \xi, \eta}
& =
\sca{T \rho(u_{\nu_{k}, j_{k}} \dots u_{\nu_1, j_1}) \xi \otimes e_\nu, \eta \otimes e_\mu} \\
& =
\sca{T L_{\nu} \rho(u_{\nu_k, j_{k}} \dots u_{\nu_1, j_1}) \xi \otimes e_\mt, \eta \otimes e_\mu} \\
& =
\sca{L_{\nu} \rho(u_{\nu_k, j_{k}} \dots u_{\nu_1, j_1}) T \xi \otimes e_\mt, \eta \otimes e_\mu} \\
& =
\sca{\rho(u_{\nu_k, j_{k}} \dots u_{\nu_1, j_1}) T \xi \otimes e_\mt, (L_{\nu} )^*\eta \otimes e_\mu}
=
0.
\end{align*}
Therefore by summing over the $j_i$ we obtain
\begin{align*}
T_{\mu, \nu}
& =
\sum_{j_k \in [n_{\nu_k}]} \cdots \sum_{j_1 \in [n_{\nu_1}]} T_{\mu, \nu} u_{\nu_k, j_{k}} \dots u_{\nu_1, j_1} v_{\nu_1, j_1} \dots v_{\nu_{k}, j_{k}}
=
0
\end{align*}
so that $T$ is right lower triangular.
We thus check the non-negative Fourier co-efficients.
For $m=0$ we have that $T_{(0)}$ commutes with $\rho(\A')$ and therefore every $T_{w,w}$ is in $\A''$.
Moreover for $w \in \bF_+^d$ with $w = w_k \dots w_1$ we have that
\begin{align*}
T_{w,w} u_{w_{k}, j_{k}} \cdots u_{w_1, j_1} \xi \otimes e_w
& =
G_0(T) L_w \rho(u_{w_{k}, j_{k}}) \cdots \rho(u_{w_1, j_1}) \xi \otimes e_\mt \\
& =
L_w \rho(u_{w_{k}, j_{k}}) \cdots \rho(u_{w_1, j_1}) G_0(T) \xi \otimes e_\mt \\
& =
u_{w_{k}, j_{k}} \cdots u_{w_1, j_1} T_{\mt, \mt} \xi \otimes e_w.
\end{align*}
Consequently we obtain
\begin{align*}
\al_{w}(T_{\mt, \mt})
& =
\al_{w_k} \cdots \al_{w_1}(T_{\mt, \mt}) \\
& =
\sum_{j_k \in [n_{w_k}]} \cdots \sum_{j_1 \in [n_{w_1}]} u_{w_{k}, j_{k}} \cdots u_{w_1, j_1} T_{\mt, \mt} v_{w_1, j_1} \cdots v_{w_{k}, j_{k}} \\
& =
T_{w, w} \sum_{j_k \in [n_{w_k}]} \cdots \sum_{j_1 \in [n_{w_1}]} u_{w_{k}, j_{k}} \cdots u_{w_1, j_1} v_{w_1, j_1} \cdots v_{w_{k}, j_{k}}
=
T_{w, w}.
\end{align*}
Thus we have that $G_0(T) = \pi(T_{\mt, \mt})$.
Now let $m >0$ and use that $G_m(T)$ commutes with $L_{i} \rho(u_{i, j_i})$ to deduce that
\begin{align*}
T_{(\mu)} L_{i} \rho(u_{i, j_i})
& =
R_\mu^* G_m(T) L_{i} \rho(u_{i, j_i})
=
R_\mu^* L_{i} \rho(u_{i, j_i}) G_m(T).
\end{align*}
However for $\xi \otimes e_\nu \in \K$ we have that
\begin{align*}
(R_\mu)^* L_{i} \rho(u_{i, j_i}) G_m(T) \xi \otimes e_\nu
& =
u_{i, j_i} T_{\nu \ol{\mu}, \nu} \xi \otimes (\Br_\mu)^* e_{i \nu \mu}
=
L_{i} \rho(u_{i, j_i}) T_{(\mu)} \xi \otimes e_\nu
\end{align*}
which yields that $T{(\mu)}$ commutes with every $L_{i} \rho(u_{i, j_i})$.
Furthermore for $y \in \A'$ we get that
\begin{align*}
T_{(\mu)} \rho(y)
& =
(R_\mu)^* G_m(T) \rho(y)
=
(R_\mu)^* \rho(y) G_m(T) \\
& =
\rho(y) (R_\mu)^* G_m(T)
=
\rho(y) T_{(\mu)}.
\end{align*}
Therefore $T_{(\mu)}$ is a diagonal operator in $(\scp{\A'}{\al}{\L_d})'$ and thus $T_{(\mu)} = \pi(T_{\ol{\mu}, \mt})$ by what we have shown for the zero Fourier co-efficients.
This shows that $G_m(T)$ is in $\scp{\A''}{\al}{\R_d}$ for all $m \in \bZ_+$.

The other equalities follow in a similar way and are left to the reader.
\end{proof}

Recall that $\A$ is \emph{inverse closed} if $\A^{-1} \subseteq \A$.
It is well known that every commutant is automatically inverse closed.

\begin{corollary}\label{C:bicom cun}
Let $(\A, \{\al_i\}_{i \in [d]})$ be a w*-dynamical system of a uniformly bounded spatial action.
Then the following are equivalent
\begin{enumerate}
\item $\A$ has the bicommutant property;
\item $\scp{\A}{\al}{\L_d}$ has the bicommutant property;
\item $\scp{\A}{\al}{\R_d}$ has the bicommutant property;
\item $\A \ol{\otimes} \L_d$ has the bicommutant property;
\item $\A \ol{\otimes} \R_d$ has the bicommutant property.
\end{enumerate}
If any of the items above hold then all algebras are inverse closed.
\end{corollary}

\begin{proof}
We just comment that the equivalence between items (i) and (ii) follows by using $(\scp{\A}{\al}{\L_d})'' = \scp{\A''}{\al}{\L_d}$ from Theorem \ref{T:com cun} and applying the compression to the $(\mt,\mt)$-entry.
\end{proof}

\begin{corollary}\label{C:inv cl cun}
(i) Let $\{\al_i\}_{i \in [d]}$ be a uniformly bounded spatial action on $\B(\H)$.
Then the w*-semicrossed products $\scp{\B(\H)}{\al}{\L_d}$ and $\scp{\B(\H)}{\al}{\R_d}$ are inverse closed.

(ii) Let $(\A, \{\al_i\}_{i \in [d]})$ be an automorphic system over a maximal abelian selfadjoint algebra (m.a.s.a.) $\A$.
Then the w*-semicrossed products $\scp{\A}{\al}{\L_d}$ and $\scp{\A}{\al}{\R_d}$ are inverse closed.
\end{corollary}

\begin{proof}
Notice that in both cases $\A = \B'$ for a suitable $\B$ and that $\scp{\B}{u}{\L_d}$ and $\scp{\B}{u}{\R_d}$ are well defined.
The proof then follows by writing $\scp{\A}{\al}{\L_d} = (\scp{\B}{u}{\R_d})'$ and the symmetrical $\scp{\A}{\al}{\R_d} = (\scp{\B}{u}{\L_d})'$.
\end{proof}

\subsection{Semicrossed products over $\bZ_+^d$}\label{Ss:scp com bicom}

Recall the decomposition in Proposition \ref{P:disintegrate}.
By applying Theorem \ref{T:com cun} recursively we obtain the following theorem.

\begin{theorem}\label{T:com com}
Let $(\A, \al, \bZ_+^d)$ be a unital w*-dynamical system.
Suppose that each $\al_{\Bi}$ is implemented by a uniformly bounded row operator $u_{\Bi}$.
Then
\[
(\scp{\A}{\al}{\bZ_+^d})'
\simeq
\left( \cdots \left( ( \scp{\A'}{u_{\bo{1}}}{\bZ_+} ) \ol{\times}_{\wh{u}_{\bo{2}}} \bZ_+ \right) \cdots \right) \ol{\times}_{\wh{u}_{\bo{d}}} \bZ_+
\]
where $\wh{u}_{\bo{i}} = u_{\bo{i}} \otimes^{(i-1)} I_{\ell^2}$ for $i = 2, \dots, d$.
\end{theorem}

Consequently we obtain the following corollaries.
Their proofs follow as in the free semigroup case and are omitted. 

\begin{corollary}\label{C:bicom com}
Let $(\A, \al, \bZ_+^d)$ be a unital w*-dynamical system.
Suppose that each $\al_{\Bi}$ is implemented by a uniformly bounded row operator $u_{\Bi}$.
Then the following are equivalent
\begin{enumerate}
\item $\A$ has the bicommutant property;
\item $\scp{\A}{\al}{\bZ_+^d}$ has the bicommutant property;
\item $\A \, \ol{\otimes} \, \bH^{\infty}(\bZ_+^d)$ has the bicommutant property.
\end{enumerate}
If any of the items above hold then all algebras are inverse closed.
\end{corollary}

\begin{corollary}\label{C:inv cl com}
(i) Let $(\B(\H), \al, \bZ_+^d)$ be a w*-dynamical system such that each $\al_{\Bi}$ is implemented by a uniformly bounded row operator $u_{\Bi}$.
Then the w*-semicrossed product $\scp{\B(\H)}{\al}{\bZ_+^d}$ is inverse closed.

(ii) Let $(\A, \al, \bZ_+^d)$ be an automorphic system over a maximal abelian selfadjoint algebra (m.a.s.a) $\A$.
Then the w*-semicrossed product $\scp{\A}{\al}{\bZ_+^d}$ is inverse closed. 
\end{corollary}

\section{Reflexivity}\label{S:reflexivity}

\subsection{Semicrossed products over $\bF_+^d$}\label{Ss:scp free}

Let $(\B(\H), \{\al_i\}_{i \in [d]})$ be a unital w*-dynamical system of a uniformly bounded spatial action such that each $\al_i$ is implemented by 
\[
u_i = [u_{i,j_i}]_{j_i \in [n_i]}.
\]
We aim to show that $\scp{\B(\H)}{\al}{\L_d}$ is similar to $\B(\H) \, \ol{\otimes} \, \L_N$ for $N = \sum_i n_i$.
Recall that we write
\[
\{(i ,j_i) \mid j_i \in [n_i], i \in [d]\}
\]
for the generators of $\bF_+^N$, i.e. we see $\bF_+^N$ as the free product $\ast_{i \in [d]} \bF_+^{n_i}$.
To this end we define the operator
\[
U \colon \H \otimes \ell^2(\bF_+^N) \to \H \otimes \ell^2(\bF_+^d)
\]
by
$U \xi \otimes e_\mt = \xi \otimes e_\mt$
and
\[
U \xi \otimes e_{(\mu_k, j_k) \dots (\mu_1, j_1)} = u_{\mu_1, j_1} \cdots u_{\mu_k, j_k} \xi \otimes e_{\mu_k \dots \mu_1}.
\]
For words of length $k$ we define the spaces
\[
\K_k := \ol{\spn}\{ \xi \otimes e_{(\mu_k, j_k) \dots (\mu_1, j_1)} \mid \xi \in \H, (\mu_i, j_i) \in ([d], [n_{\mu_i}]) \}.
\]
The ranges of $\K_k$ under $U$ are orthogonal and thus 
\begin{align*}
\nor{U|_{\K_k}} 
& =
\sup_{|\mu| = k} \| u_{\mu_1} \cdot (u_{\mu_2} \otimes I_{[n_{\mu_1}]}) \cdots (u_{\mu_k} \otimes I_{[n_{\mu_1} \cdots n_{\mu_{k-1}}]}) \|
=
\sup_{|\mu| = k} \nor{\wh{u}_\mu}
\end{align*}
which is bounded (by the uniform bound for $\{u_i\}_{i \in [d]}$).
As $U = \oplus_k U|_{\K_k}$ we derive that $U$ is bounded.
In particular the operator $U$ is invertible with 
\[
U^{-1} \colon \H \otimes \ell^2(\bF_+^d) \to \H \otimes \ell^2(\bF_+^N)
\]
given by
$U^{-1} \xi \otimes e_\mt = \xi \otimes e_\mt$
and
\[
U^{-1} \xi \otimes e_{\mu_k \dots \mu_1} = \sum_{j_1 \in [n_{\mu_1}]} \cdots \sum_{j_k \in [n_{\mu_k}]} v_{\mu_k, j_k} \cdots v_{\mu_1, j_1} \xi \otimes e_{(\mu_k, j_k) \dots (\mu_1, j_1)}
\]
where $v_i$ is the inverse of $u_i$.
Notice that if $K$ is the uniform bound for $\{\wh{u}_{\mu}\}_{\mu}$ and $\{\wh{v}_{\mu}\}_{\mu}$ then $\max\{ \nor{U}, \nor{U^{-1}}\} = K$.

\begin{theorem}\label{T:ue LN}
Let $(\B(\H), \{\al_i\}_{i \in [d]})$ be a w*-dynamical system of a uniformly bounded spatial action.
Suppose that every $\al_i$ is given by an invertible row operator $u_i = [u_{i, j_i}]_{j_i \in [n_i]}$ and set $N = \sum_{i \in [d]} n_i$.
Then the w*-semicrossed product $\scp{\B(\H)}{\al}{\L_d}$ is similar to $\B(\H) \, \ol{\otimes} \, \L_{N}$ . 
\end{theorem}

\begin{proof}
We will show that the constructed $U$ yields the required similarity.
To this end we apply for $x \in \B(\H)$ to obtain
\begin{align*}
\ol{\pi}(x) U \xi \otimes e_{(\mu_k, j_k) \dots (\mu_1, j_1)}
& =
\al_{\mu_1} \cdots \al_{\mu_k}(x) u_{\mu_1, j_1} \cdots u_{\mu_k, j_k} \xi \otimes e_{\mu_k \dots \mu_1} \\
& =
u_{\mu_1, j_1} \cdots u_{\mu_k, j_k} x \xi \otimes e_{\mu_k \dots \mu_1} \\
& =
U \rho(x) \xi \otimes e_{(\mu_k, j_k) \dots (\mu_1, j_1)}
\end{align*}
where we used that $\al_{\mu_i}(x) u_{\mu_i, j_i} = u_{\mu_i, j_i} x$.
On the other hand we have that
\begin{align*}
L_i U \xi \otimes e_{(\mu_k, j_k) \dots (\mu_1, j_1)}
& =
L_i u_{\mu_1, j_1} \cdots u_{\mu_k, j_k} \xi \otimes e_{\mu_k \dots \mu_1} \\
& =
u_{\mu_1, j_1} \cdots u_{\mu_k, j_k} \xi \otimes e_{i \mu_k \dots \mu_1}
\end{align*}
whereas
\begin{align*}
U \sum_{j_i \in [n_i]} L_{i, j_i} \rho(v_{i, j_i}) \xi \otimes e_{(\mu_k, j_k) \dots (\mu_1, j_1)}
& = \\
& \hspace{-3cm} = 
U \sum_{j_i \in [n_i]} v_{i, j_i} \xi \otimes e_{(i,j _i) (\mu_k, j_k) \dots (\mu_1, j_1)} \\
& \hspace{-3cm} = 
\sum_{j_i \in  [n_i]} u_{\mu_1, j_1} \dots u_{\mu_k, j_k} u_{i, j_i} v_{i, j_i} \xi \otimes e_{i \mu_k \dots \mu_1} \\
& \hspace{-3cm} = 
u_{\mu_1, j_1} \dots u_{\mu_k, j_k} \xi \otimes e_{i \mu_k \dots \mu_1}
\end{align*}
since $\sum_{j_i \in [n_i]} u_{i, j_i} v_{i, j_i} = I$.
Hence we obtain that
\[
U^{-1} L_i U = \sum_{j_i \in [n_i]} L_{i, j_i} \rho(v_{i, j_i}) \foral i \in [d].
\]
Therefore the generators of $\scp{\B(\H)}{\al}{\L_d}$ are mapped into $\B(\H) \, \ol{\otimes} \, \bF_+^N$.
We need to show that the elements $\rho(x)$ and $U^{-1} L_i U$ also generate the elements 
\[
L_{i, j_{i}} \foral (i, j_{i}) \in ([d], [n_{i}]).
\]
Since every $u_{i , j_{i}}$ is in $\B(\H)$ we have that
\[
U^{-1} L_i U \rho(u_{i , j_{i}}) = \sum_{j_i \in [n_i]} L_{i, j_i} \rho(v_{i, j_i}) \rho(u_{i , j_{i}}) = L_{i, j_{i}}
\]
and the proof is complete.
\end{proof}

\begin{theorem}\label{T:cun ref}
Let $(\A, \{\al_i\}_{i \in [d]})$ be a w*-dynamical system of a uniformly bounded spatial action.
Suppose that every $\al_i$ is given by an invertible row operator $u_i = [u_{i, j_i}]_{j_i \in [n_i]}$ and set $N = \sum_{i \in [d]} n_i$.

(i) If $N \geq 2$ then every w*-closed subspace of $\scp{\A}{\al}{\L_d}$ or $\scp{\A}{\al}{\R_d}$ is  hyper-reflexive.
If $K$ is the uniform bound related to $\{u_i\}$ then the hyper-reflexivity constant is at most $3 \cdot K^4$.

(ii) If $N =1$ and $\A$ is reflexive then $\scp{\A}{\al}{\L_d} = \scp{\A}{\al}{\R_d} = \scp{\A}{\al}{\bZ_+}$ is reflexive.
\end{theorem}

\begin{proof}
If every $\al_i$ is implemented by an invertible row operator $u_i$ then $(\A, \{\al_i\}_{i \in [d]})$ extends to $(\B(\H), \{\al_i\}_{i \in [d]})$ so that
\[
\scp{\A}{\al}{\L_d} \subseteq \scp{\B(\H)}{\al}{\L_d} \simeq \B(\H) \, \ol{\otimes} \, \L_N
\]
by Theorem \ref{T:ue LN}.
If $N \geq 2$ then every w*-closed subspace of $\B(\H) \, \ol{\otimes} \, \L_N$ is hyper-reflexive with distance constant at most $3$ by \cite{Ber98}.
As hyper-reflexivity is preserved under taking similarities the proof of item (i) is complete.
Item (ii) follows by \cite[Theorem 2.9]{Kak09}.
\end{proof}

\begin{corollary}\label{C:cun ref}
Let $(\A, \{\al_i\}_{i \in [d]})$ be a w*-dynamical system so that every $\al_i$ is given by a Cuntz family $[s_{i, j_i}]_{j_i \in [n_i]}$.
If $N = \sum_{i \in [d]} n_i \geq 2$ then every w*-closed subspace of $\scp{\A}{\al}{\L_d}$ or $\scp{\A}{\al}{\R_d}$ is hyper-reflexive with distance constant at most $3$.
\end{corollary}


\begin{corollary}\label{C:masa ref}
Let $(\A, \{\al_i\}_{i \in [d]})$ be a system of w*-continuous automorphisms on a maximal abelian selfadjoint algebra $\A$.
Then $\scp{\A}{\al}{\L_d}$ and $\scp{\A}{\al}{\R_d}$ are reflexive.
\end{corollary}

\begin{remark}\label{R:ref ind}
When $\A$ is reflexive, we can have an independent proof of reflexivity of $\scp{\A}{\al}{\L_d}$ that does not go through hyper-reflexivity.
First note that if an operator $T$ is in $\Ref(\scp{\A}{\al}{\L_d})$ then $T$ is left lower triangular and $T_{\mu w, w} \in \Ref(\A)$ for every $\mu, w \in \bF_+^d$.
Indeed for $\xi, \eta \in \H$ and $\nu, \nu' \in \bF_+^d$ there is a sequence $F_n \in \scp{\A}{\al}{\L_d}$ such that 
\begin{align*}
\sca{T_{\nu', \nu}\xi, \eta}
& =
\sca{T \xi \otimes e_\nu, \eta \otimes e_{\nu'}} \\
& =
\lim_n \sca{F_n \xi \otimes e_\nu, \eta \otimes e_{\nu'}}
=
\lim_n \sca{[F_n]_{\nu', \nu}\xi, \eta}.
\end{align*}
Taking $\nu \not<_l \nu'$ gives that $T$ is left lower triangular as all $F_n$ are so.
Taking $\nu' = \mu \nu$ yields $[F_n]_{\mu \nu, \nu} \in \A$ and thus $T_{\mu \nu, \nu} \in \Ref(\A)$.
Now if $\{\al_i\}_{i \in [d]}$ is a uniformly bounded spatial action then $T \in \scp{\B(\H)}{\al}{\L_d}$.
Therefore $T$ is left lower triangular and for $m \in \bZ_+$  we have that $G_m(T) = \sum_{|\mu| = m} L_\mu \ol{\pi}(T_{\mu, \mt})$ with $T_{\mu, \mt} \in \Ref(\A) = \A$.
\end{remark}

\begin{remark}
Even though reflexivity of $\A$ directly implies reflexivity of the w*-semicrossed products the converse does not hold.

For example suppose that each $\al_i$ is implemented by a single invertible $u_i$.
Then we can extend $(\A, \{\al_i\}_{i \in [d]})$ to the system $(\Ref(\A), \{\al_i\}_{i \in [d]})$.
If $d \geq 2$ then both $\scp{\A}{\al}{\L_d}$ and $\scp{\Ref(\A)}{\al}{\L_d}$ are reflexive and
\[
\scp{\A}{\al}{\L_d} \subseteq \scp{\Ref(\A)}{\al}{\L_d}.
\]
This inclusion is proper when $\A$ is not reflexive, e.g. for $\A = \{aI + b E_{21} \mid a, b \in \bC\}$ in $\M_2(\bC)$.
In fact by taking the compression to the $(\mt,\mt)$-entry we see that $\scp{\A}{\al}{\L_d} = \scp{\Ref(\A)}{\al}{\L_d}$ if and only if $\A = \Ref(\A)$.
\end{remark}

The reflexivity results extend to systems over any factor.
This can be achieved by following the ingenious arguments of Helmer \cite{Hel14}.
Even though these are originally presented in \cite{Hel14} for Type II or III factors they apply as long as two basic properties are satisfied.
We isolate these below.

\begin{definition}
An algebra $\A \subseteq \B(\H)$ is \emph{injectively reducible} if there is a non-trivial reducing subspace $M$ of $\A$ such that the representations
\[
a \mapsto a|_{M} \qand a \mapsto a|_{M^\perp}
\]
are both injective.
\end{definition}

\begin{definition}
A w*-dynamical system $(\A, \{\al_i\}_{i \in [d]})$ is \emph{injectively reflexive} if: (i) $\A$ is reflexive; (ii) $\A$ is injectively reducible by some $M$; and (iii) $\be_\nu(\A)$ is reflexive for all $\nu \in \bF_+^d$ with
\[
\be_\nu(a) = \begin{bmatrix} a|_{M} & 0 \\ 0 & \al_\nu(a)|_{M^\perp} \end{bmatrix}.
\]
\end{definition}

It is immediate that dynamical systems over Type II or Type III factors are injectively reflexive.

\begin{theorem}\label{T:Hel} \cite[Theorem 3.18]{Hel14}
If $(\A, \{\al_i\}_{i \in [d]})$ is an injectively reflexive unital w*-dynamical system then $\scp{\A}{\al}{\L_d}$ and $\scp{\A}{\al}{\R_d}$ are reflexive.
\end{theorem}

\begin{proof}
The left version is \cite[Theorem 3.18]{Hel14} after translating from the W*-cor\-respondences terminology.
To exhibit this we will show how the right case can be shown in our context.

Fix $T \in \Ref(\scp{\A}{\al}{\R_d})$.
If $m<0$ then $G_m(T) = 0$ by Remark \ref{R:ref ind}.
If $m \geq 0$ then $T_{\mu,\mt} \in \A$ by the same remark.
Thus it suffices to show that $T_{\nu \ol{\mu},\nu} = \al_\nu(T_{\mu,\mt})$ for every $\nu \in \bF_+^d$.
By assumption let $M$ be the subspace that injectively reduces $\A$.
We henceforth fix a word $\nu \in \bF_+^d$ and we define the subspaces of $\K$
\[
\K_0 := \ol{\spn}\{\xi \otimes e_w \mid \xi \in M, w \in \bF_+^d\}
\]
and
\[
\K_\nu := \ol{\spn}\{\eta \otimes e_{\nu w} \mid \eta \in M^\perp, w \in \bF_+^d\}.
\]
Both $\K_0$ and $\K_\nu$ are invariant subspaces of $\scp{\A}{\al}{\R_d}$.
If $p$ is the projection on $\K_0 \oplus K_\nu$ then we have that $G_m(T) p \in \Ref( (\scp{\A}{\al}{\R_d}) p)$.
We will use the unitary
\[
U \colon p\K \to \K: \xi \otimes e_w + \eta \otimes e_{\nu w} \mapsto (\xi + \eta) \otimes e_w.
\]
A straightforward computation shows that
\[
U \pi(a) p U^* = \sum_{w \in \bF_+^d} (\al_w(a)|_{M} + \al_{\nu w}(a)|_{M^\perp}) \otimes p_w
\]
and that $U R_i p U^* = R_i$.
In particular $p$ is reducing for $R_i$ and we get
\[
U G_m(T) p U^* = 
\sum_{|\mu| = m} \sum_{w \in \bF_+^d} R_\mu (T_{w \ol{\mu}, w}|_{M} + T_{\nu w \ol{\mu}, \nu w}|_{M^\perp}) \otimes p_w.
\]
By taking compressions we thus have that the $(\ol{\mu}, \mt)$-entry of the operator $U G_m(T) p U^*$  is in the reflexive cover of the $(\ol{\mu}, \mt)$-block of the algebra $\Ref( U (\scp{\A}{\al}{\R_d}) p U^*)$.
However the latter coincides with (the reflexive cover of, and hence with) $\be_\nu(\A)$ defined above.
Hence there is an $a \in \A$ such that
\[
T_{\ol{\mu}, \mt}|_{M} + T_{\nu \ol{\mu}, \nu}|_{M^\perp} = a|_{M} + \al_{\nu}(a)|_{M^\perp}.
\]
Since the restrictions to $M$ and $M^\perp$ are injective we derive that $T_{\ol{\mu}, \mt} = a$ and $T_{\nu \ol{\mu}, \nu} = \al_\nu(a) = \al_\nu(T_{\ol{\mu}, \mt})$, which completes the proof.
\end{proof}

By combining Theorem \ref{T:cun ref} with Theorem \ref{T:Hel} we get the next corollary.

\begin{corollary}\label{C:factor ref}
Let $(\A, \{\al_i\}_{i \in [d]})$ be a unital w*-dynamical system on a factor $\A \subseteq \B(\H)$ for a separable Hilbert space $\H$.
Then $\scp{\A}{\al}{\L_d}$ and $\scp{\A}{\al}{\R_d}$ are reflexive.
\end{corollary}

\begin{proof}
We have that either $\A = \B(\H)$ or there is a non-trivial projection $p \in \A'$, and so the system is injectively reflexive.
\end{proof}

\subsection{Semicrossed products over $\bZ_+^d$}\label{Ss:scp com}

We now pass to the examination of $\bZ_+^d$.
When every $\al_{\Bi}$ is  given by an invertible row operator $u_{\Bi} = [u_{i, j_i}]_{j_i \in [n_i]}$ then we write $M = \prod_{i \in [d]} n_i$ for the capacity of the system.
Note that $M \geq 2$ if and only if there is at least one $\Bi$ such that $n_i \geq 2$.

\begin{theorem}\label{T:com ref}
Let $(\A, \al, \bZ_+^d)$ be a unital w*-dynamical system.
Suppose that every $\al_{\Bi}$ is uniformly bounded spatial, given by an invertible row operator $u_{\Bi} = [u_{i, j_i}]_{j_i \in [n_i]}$, and set $M = \prod_{i \in [d]} n_i$.

(i) If $M \geq 2$ then every w*-closed subspace of $\scp{\A}{\al}{\bZ_+^d}$ is hyper-reflexive.
If $K_i$ is the uniform bound associated to $u_{\Bi}$ (and its inverse) then the hyper-reflexivity constant is at most $3 \cdot K^4$ for $K = \min\{K_i \mid n_i \geq 2\}$.

(ii) If $M=1$ and $\A$ is reflexive then $\scp{\A}{\al}{\bZ_+^d}$ is reflexive.
\end{theorem}

\begin{proof}
For item (i), suppose without loss of generality that $n_d \geq 2$ with $K_d = \min\{K_i \mid n_i \geq 2\}$.
Then we can write $\scp{\A}{\al}{\bZ_+^d} \simeq \scp{\B}{\wh{\al}_{\bo{d}}}{\bZ_+}$ for an appropriate w*-closed algebra $\B$ by Proposition \ref{P:disintegrate}.
Hence we can apply Theorem \ref{T:cun ref} for the system $(\B, \wh{\al}_{\bo{d}}, {\bZ_+})$, as its capacity is greater than $2$.
For item (ii) we can write $\scp{\A}{\al}{\bZ_+^d}$ as successive w*-semicrossed products and apply recursively \cite[Theorem 2.9]{Kak09}, i.e. Theorem \ref{T:cun ref}(ii).
\end{proof}

\begin{corollary}\label{C:com ref}
Let $(\A, \al, \bZ_+^d)$ be a unital w*-dynamical system.
Suppose that at least one $\al_{\Bi}$ is implemented by a Cuntz family $[s_{i, j_i}]_{j_i \in [n_i]}$ with $n_i \geq 2$.
Then every w*-closed subspace of $\scp{\A}{\al}{\bZ_+^d}$ is hyper-reflexive with distance constant $3$.
\end{corollary}

\begin{proof}
Suppose without loss of generality that $\al_{\bo{d}}$ is defined by a Cuntz family with $n_{\bo{d}} \geq 2$.
Then $\wh{\al}_{\bo{d}}$ is also given by the Cuntz family $\{s_{j} \otimes^{d-1} I\}$ of size $n_{\bo{d}}$.
By  Proposition \ref{P:disintegrate} we can write $\scp{\A}{\al}{\bZ_+^d} \simeq \scp{\B}{\wh{\al}_{\bo{d}}}{\bZ_+}$ for some w*-closed algebra $\B$. 
Applying then Corollary \ref{C:cun ref} completes the proof.
\end{proof}

\begin{corollary}\label{C:tensor com ref}
If $\A$ is reflexive then $\A \, \ol{\otimes} \, \bH^\infty(\bZ_+^d)$ is reflexive.
\end{corollary}

\begin{corollary}\label{C:masa com ref}
Let $(\A,\al,\bZ_+^d)$ be a unital automorphic system over a maximal abelian selfadjoint algebra $\A$.
Then $\scp{\A}{\al}{\bZ_+^d}$ is reflexive.
\end{corollary}

We can apply the arguments of \cite{Hel14} to tackle other dynamical systems.

\begin{definition}
A w*-dynamical system $(\A,\al,\bZ_+^d)$ is \emph{injectively reflexive} if: (i) $\A$ is reflexive, (ii) $\A$ is injectively reducible by $M$; and (iii) $\be_{\un{n}}(\A)$ is reflexive for all $\un{n} \in \bZ_+^d$ with
\[
\be_{\un{n}}(a) = \begin{bmatrix} a|_{M} & 0 \\ 0 & \al_{\un{n}}(a)|_{M^\perp} \end{bmatrix}.
\]
\end{definition}

Consequently every $(\A,\al_\Bi,\bZ_+)$ is injectively reflexive for the same $M$.
Again it follows that systems over Type II or Type III factors are injectively reflexive.

\begin{theorem}\label{T:Hel com}
Let $(\A,\al,\bZ_+^d)$ be a unital w*-dynamical system.
If the system is injectively reflexive then $\scp{\A}{\al}{\bZ_+^d}$ is reflexive.
\end{theorem}

\begin{proof}
The proof follows in a similar way as in Theorem \ref{T:Hel}.
In short if $T$ is in $\Ref(\scp{\A}{\al}{\bZ_+^d})$ then $T$ is lower triangular and $T_{\un{m}, \un{0}} \in \A$ for every $\un{m} \in \bZ_+^d$.
Thus we just need to show that $T_{\un{m} + \un{n}, \un{n}} = \al_{\un{n}}(T_{\un{m}, \un{0}})$ for every $\un{n} \in \bZ_+^d$.
For a fixed $\un{n}$ let the spaces
\[
K_{\un{0}} := \ol{\spn}\{ \xi \otimes e_{\un{w}} \mid \xi \in M, \un{w} \in \bZ_+^d\}
\]
and
\[
K_{\un{n}} := \ol{\spn}\{ \eta \otimes e_{\un{n} + \un{w}} \mid \eta \in M^\perp, \un{w} \in \bZ_+^d\}
\]
and let the unitary $U \colon K_{\un{0}} \oplus K_{\un{n}} \to \H \otimes \ell^2(\bZ_+^d)$ given by
\[
U(\xi \otimes e_{\un{w}} + \eta \otimes e_{\un{n} + \un{w}}) = (\xi + \eta) \otimes e_{\un{w}}.
\]
If $p$ is the projection on $K_{\un{0}} \oplus K_{\un{n}}$ then
\[
U \pi(a) p U^* = \sum_{\un{w} \in \bZ_+^d} (\al_{\un{w}}(a)|_{M} + \al_{\un{n} + \un{w}}(a)|_{M^\perp}) \otimes p_{\un{w}} \qand U L_{\Bi}p U^* = L_{\Bi}.
\]
On the other hand we have that
\[
U G_{\un{m}}(T)p U^* = L_{\un{m}} \sum_{\un{w} \in \bZ_+^d} (T_{\un{m} + \un{w}, \un{w}}|_{M} + T_{\un{n} + \un{m} + \un{w}, \un{n} + \un{w}}|_{M^\perp}) \otimes p_{\un{w}}.
\]
Taking compressions and using reflexivity of $\be_{\un{n}}(\A)$ implies that there exists an $a \in \A$ such that
\[
T_{\un{m}, \un{0}}|_M + T_{\un{n} + \un{m}, \un{n}}|_{M^\perp} = a|_M + \al_{\un{n}}(a)|_{M^\perp},
\]
and therefore $T_{\un{m} + \un{n}, \un{n}} = \al_{\un{n}}(a) = \al_{\un{n}}(T_{\un{m}, \un{0}})$.
\end{proof}

\begin{corollary}\label{C:factor com ref}
Let $(\A,\al,\bZ_+^d)$ be a unital w*-dynamical system on a factor $\A \subseteq \B(\H)$ for a separable Hilbert space $\H$.
Then $\scp{\A}{\al}{\bZ_+^d}$ is reflexive.
\end{corollary}

\begin{remark}
The w*-semicrossed products $\scp{\A}{\al}{\bZ_+^d}$ do not fit in the theory of W*-correspondences.
This has been observed in \cite{DFK14, Kak13} for the norm-analogues but the arguments apply here mutatis mutandis.
That is, when $\A = \bC$ then $\scp{\A}{\al}{\bZ_+^d}$ is the commutative algebra $\bH^\infty(\bZ_+^d)$.
Therefore the results of this section are disjoint from those of \cite{Hel14} when $d \geq 2$.
\end{remark}



\end{document}